\documentclass[11pt]{amsart}
\pdfoutput=1

\title{End spaces and tree-decompositions}
\author{Marcel Koloschin}
\author{Thilo Krill}
\author{Max Pitz}
\address{Universit\"at Hamburg, Department of Mathematics, Bundesstrasse 55 (Geomatikum), 20146 Hamburg, Germany}
\email{\{marcel.koloschin, thilo.krill, max.pitz\}@uni-hamburg.de}
\usepackage{amsmath,amssymb,amsthm}  
\usepackage{mathtools}
\usepackage{tikz, tikz-cd}
\usetikzlibrary{shapes.geometric}
\usetikzlibrary{shapes,decorations.pathmorphing,backgrounds,positioning,fit,petri, hobby}
\usepackage{comment}
\usepackage{enumerate}
\usepackage{enumitem}
\usepackage{mathrsfs}

\usepackage{xcolor} 	
\usepackage[unicode]{hyperref}
\hypersetup{
	colorlinks,
    linkcolor={red!60!black},
    citecolor={green!60!black},
    urlcolor={blue!60!black},
}
\usepackage[abbrev, msc-links]{amsrefs} 

\usepackage[utf8]{inputenc}
\usepackage[T1]{fontenc}
\usepackage{lmodern}
\usepackage[babel]{microtype}
\usepackage[english]{babel}

\usepackage{geometry}
\usepackage{enumitem}

\let\polishlcross=\l
\def\l{\ifmmode\ell\else\polishlcross\fi}

\let\emptyset=\varnothing

\let\eps=\varepsilon
\let\theta=\vartheta
\let\rho=\varrho
\let\phi=\varphi

\def\NN{\mathbb N}

\def\cC{{\mathscr C}}

\def\cP{{\mathcal P}}

\def\cE{{\mathcal E}}
\def\cR{{\mathcal R}}

\def\cT{{\mathcal T}}

\def\cV{{\mathcal V}}

\def\cO{{\mathcal O}}
 
\defˆ{^}
\newcommand{\script}{\mathcal}
\newcommand{\parentheses}[1]{{\left( {#1} \right)}}

\newcommand{\p}{\parentheses}

\newcommand{\Set}[1]{{\left\lbrace {#1} \right\rbrace}}

\def\set#1:#2{\Set{{#1} \colon {#2}}}

\theoremstyle{plain}
\newtheorem{thm}{Theorem}[section]

\newtheorem{clm}[thm]{Claim}
\newtheorem{cor}[thm]{Corollary}
\newtheorem{lemma}[thm]{Lemma}

\newtheorem{problem}[thm]{Problem}

\theoremstyle{definition}
\newtheorem{exmp}[thm]{Example}

\begin{document}

\begin{abstract}
We present a systematic investigation into how tree-decompositions of finite adhesion capture topological properties of the space formed by a graph together with its ends.
As main results, we characterise when the ends of a graph can be distinguished, and characterise which subsets of ends can be displayed by a tree-decomposition of finite adhesion.
   
In particular, we show that a subset  $\Psi$ of the ends of a graph $G$ can be displayed by a tree-decomposition of finite adhesion if and only if $\Psi$ is $G_\delta$ (a countable intersection of open sets) in $|G|$, the  topological space formed by a graph together with its ends. Since the undominated ends of a graph are easily seen to be $G_\delta$, this provides a structural explanation for Carmesin's result that the set of undominated ends can always be displayed.
\end{abstract}

\vspace*{-36pt}
\maketitle

\section{Introduction}

In this paper we settle the question up to which complexity the topological spaces $|G|$ formed by an infinite graph $G$ together with its ends can still be encoded by tree-decompositions of finite adhesion of the underlying graph $G$. 

To state our results more precisely, recall that a 
\emph{separation} of a graph $G$ is an unordered pair $\{A,B\}$ of sets of vertices in $G$ such that $A \cup B = V(G)$ and $G$ has no edge between $A \setminus B$ and $B \setminus A$, which is equivalent to saying that its \emph{separator} $A \cap B$ separates $A$ from $B$. The cardinal $|A \cap B|$ is the \emph{order} of the separation $\{A, B\}$ and the sets $A, B$ are its \emph{sides}.

    \vspace{-2pt}
\begin{figure}[h]
    \centering
    \begin{tikzpicture}[xscale=0.9]
    \tikzstyle{white node}=[draw,circle,fill=black,minimum size=4pt,inner sep=0pt]
    \tikzstyle{dot}=[draw,circle,fill=black,minimum size=0pt,inner sep=0pt]
    \tikzstyle{red node}=[draw,circle,color=red!50,fill=red!50,minimum size=7pt,inner sep=0pt]
    \tikzstyle{red circle}=[draw,circle,color=red,minimum size=75pt,inner sep=0pt]
    \tikzstyle{red edge} = [draw,line width=2pt,-,red!50]
    \tikzstyle{red dot}=[draw,circle,color=red,fill=red,minimum size=2pt,inner sep=0pt]
                                    
        \draw (0,0) node[white node] {}
        -- ++(-40:2.2cm) node[white node] (1) {}
        -- ++(230:1.6cm) node[white node] (11) {};
                \draw (1)
        -- ++(-10:1.5cm) node[white node] (2)  {}
        -- ++(0:1.8cm) node[white node] (3)  {};
            \draw (2)
        -- ++(80:1.5cm) node[dot] (21) {}
        -- ++(80:0.6cm) node[white node] (22) {};
            \draw (22)
        -- ++(40:1.3cm) node[white node] (221) {};
            \draw (22)
        -- ++(130:1.6cm) node[white node] (222)  {};
     \draw (7,-1.61) node[white node] (4) {}
        -- ++(17:1.8cm) node[white node] (5) {}
        -- ++(54:2cm) node[white node] (6) {};
    \draw (4)
        -- ++(-60:1.35cm) node[white node] (7) {};
    \draw (5)
        -- ++(-50:1.4cm) node[white node] (8) {};
        
    \draw[loosely dashed] (3) -- (4);
    \draw[rotate=-40] (0,0cm) ellipse (40pt and 20pt);
    \draw[rotate=-10] (1) ellipse (32pt and 22pt);
    \draw[rotate=230] (11) ellipse (27pt and 13pt); 
    \draw[rotate=0] (2) ellipse (31pt and 20pt);
    \draw[rotate=80] (21) ellipse (35pt and 15pt);
    \draw[rotate=40] (221) ellipse (30pt and 16pt);
    \draw[rotate=130] (222) ellipse (38pt and 16pt);
    \begin{scope}

    \draw[clip,rotate=5] (3) ellipse (34pt and 20pt);
    \draw[rotate=5,fill=gray, opacity=0.3] (4) ellipse (38pt and 25pt);
    \end{scope}
    \draw[rotate=5] (4) ellipse (38pt and 25pt);
    \draw[rotate=0] (5) ellipse (25pt and 25pt);
    \draw[rotate=50] (6) ellipse (40pt and 25pt);
    \draw[rotate=-40] (7) ellipse (25pt and 22pt);
    \draw[rotate=-50] (8) ellipse (28pt and 16pt);
    
    \node at (5.8,-2.7){$V_{t_1} \cap V_{t_2}$};
    \node at (5,-1.4){$t_1$};
    \node at (7.1,-1.3){$t_2$};
     \node at (2.5,-0.5){$U_1$};
    \node at (10,-1){$U_2$};

\end{tikzpicture}
    \vspace{-12pt}
    \caption{$V_{t_1} \cap V_{t_2}$ separates $U_1$ from $U_2$.}
    \label{fig:td}
\end{figure}
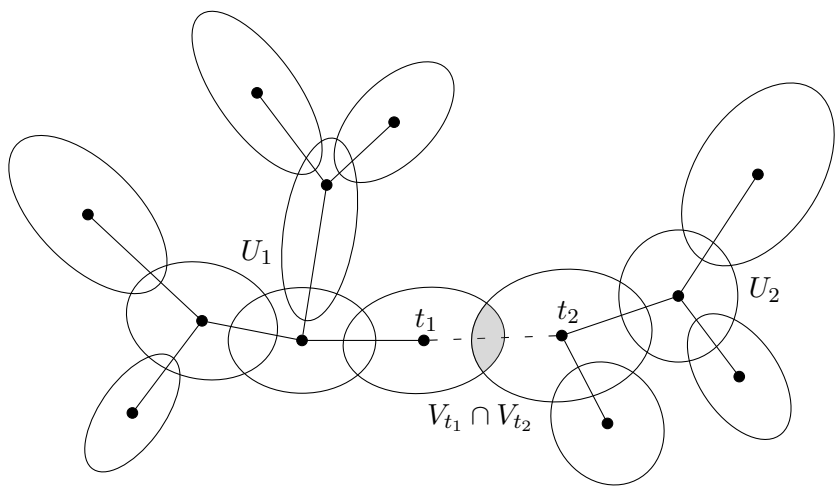

An \emph{end} of a connected (infinite) graph $G$ is an equivalence class of rays, where two rays are  equivalent if for every finite order separation of the graph $G$, the rays eventually belong both to the same side of the separation. The set of ends is denoted by $\Omega(G)$. The space $|G|$ is the topological space on $G \cup \Omega(G)$ equipped with a natural topology called \textsc{MTop}, described in detail in Section~\ref{sec_prelims}. For locally  finite connected graphs, this $|G|$ is precisely the Freudenthal compactification of $G$. A longstanding quest in graph theory is to understand end spaces of  infinite graphs that are not necessarily locally finite, cf.~\cites{diestel1992end, diestel2006end,diestel2011locally, kurkofka2021approximating,kurkofkapitz_rep,FirstSecondCountable,polat1996ends2, polat1996ends}.

A \emph{tree-decomposition} of a graph $G$ is a set $\mathcal{T}=\{V_t \colon t \in T\}$ of $G$ into subsets of $V(G)$ called \emph{parts} indexed by some tree $T$ so that $G = \bigcup_{t \in T} G[V_t]$ and the parts mirror the separation properties of the tree: just like removing any edge $e =t_1t_2$ from $T$ gives rise to two components $T_1$ and $T_2$ of $T-e$, so does removing $X_e:=V_{t_1} \cap V_{t_2}$ from $G$ separate any part of $T_1$ from any part of $T_2$, see Figure~\ref{fig:td}. More formally, writing $U_1=\bigcup \{ V_t \colon t \in T_1 \}$ and $U_2=\bigcup \{ V_t \colon t \in  T_2 \}$, we require that $\{U_1,U_2\}$ is a separation of $G$ with separator $X_e$. If all such separations are of finite order, we say the tree-decomposition has \emph{finite adhesion}.

Now consider how the ends of a graph $G$ interact with a tree-decomposition $\mathcal{T}$ of finite adhesion. As every edge $e$ of $T$ induces a finite order separation $\{A_e,B_e\}$ of $G$, any end of $G$ has to choose one side of $T-e$, and we may visualize this decision by orienting $e$ accordingly. Then for a fixed end, all the edges point either towards a unique node or towards a unique end of $T$, see Figure~\ref{fig:orient}. 
In this way, each end of $G$ \emph{lives} in a part of $\mathcal{T}$ or \emph{corresponds} to an end of $T$, and we may encode this correspondence  by a map $f_\mathcal{T}  \colon \Omega(G) \to V(T) \cup \Omega(T)$.

 \vspace{-12pt}
\begin{figure}[h]
    \centering
    \begin{tikzpicture}[xscale=0.9]
    \tikzstyle{white node}=[draw,circle,fill=black,minimum size=4pt,inner sep=0pt]
    \tikzstyle{dot}=[draw,circle,fill=black,minimum size=0pt,inner sep=0pt]
    \tikzstyle{red node}=[draw,circle,color=red!50,fill=red!50,minimum size=7pt,inner sep=0pt]
    \tikzstyle{red circle}=[draw,circle,color=red,minimum size=75pt,inner sep=0pt]
    \tikzstyle{red edge} = [draw,line width=2pt,-,red!50]
    \tikzstyle{red dot}=[draw,circle,color=red,fill=red,minimum size=2pt,inner sep=0pt]
                                    
        \draw (0,0) node[white node] (0){}
        -- ++(-40:2.2cm) node[white node] (1) {}
        -- ++(230:1.6cm) node[white node] (11) {};
        
                \draw (1)
        -- ++(-10:1.5cm) node[white node] (2)  {}
        -- ++(0:1.8cm) node[white node] (3)  {};

            \draw (2)
        -- ++(80:1.5cm) node[dot] (21) {}
        -- ++(80:0.6cm) node[white node] (22) {};
    
            \draw (22)
        -- ++(40:1.3cm) node[white node] (221) {};
        
            \draw (22)
        -- ++(130:1.6cm) node[white node] (222)  {};
        
     \draw (7,-1.61) node[white node] (4) {}
        -- ++(17:1.8cm) node[white node] (5) {}
        -- ++(54:2cm) node[white node] (6) {};
          
    \draw (4)
        -- ++(-60:1.35cm) node[white node] (7) {};
        
    \draw (5)
        -- ++(-50:1.4cm) node[white node] (8) {};
        
    \draw[->] (3) -- (4);
          \draw[->] (0) -- (1); 
          \draw[->] (11) -- (1);
        \draw[->] (1) -- (2);  
        \draw[->] (2) -- (3);
      \draw[->] (21) -- (2);
            \draw[->] (221) -- (22);
            \draw[->] (222) -- (22);
            \draw[->] (7) -- (4);
             \draw[->] (6) -- (5);
            \draw[->] (4) -- (5);
            \draw[->] (5) -- (8);
    \draw[rotate=-40] (0,0cm) ellipse (40pt and 20pt);
    \draw[rotate=-10] (1) ellipse (32pt and 22pt);
    \draw[rotate=230] (11) ellipse (27pt and 13pt); 
    \draw[rotate=0] (2) ellipse (31pt and 20pt);
     \draw[rotate=80] (21) ellipse (35pt and 15pt);
    \draw[rotate=40] (221) ellipse (30pt and 16pt);
    \draw[rotate=130] (222) ellipse (38pt and 16pt);
    
    \draw[rotate=5] (3) ellipse (34pt and 20pt);
   
    \draw[rotate=5] (4) ellipse (38pt and 25pt);
    \draw[rotate=0] (5) ellipse (25pt and 25pt);
    \draw[rotate=50] (6) ellipse (40pt and 25pt);
    \draw[rotate=-40] (7) ellipse (25pt and 22pt);
    \draw[rotate=-50] (8) ellipse (28pt and 16pt);

    \draw[purple, rounded corners, thick, ->] (4.7,-1.4)  -- ++(180:1.3cm)  -- ++(80:1.5cm) -- ++(170:0.4cm) -- ++(260:1.6cm) -- ++(170:2cm) -- ++(260:.4cm) -- ++(-10:2.3cm) -- ++(0:3.7cm)  -- ++(-60:1.5cm)  -- ++(30:0.5cm) -- ++(120:1.5cm) -- ++(15:1.3cm) -- ++(-50:1.8cm);
    \node at (10,-3.2){$\textcolor{purple}{R}$};
\end{tikzpicture}
    \vspace{-12pt}
    \caption{A ray $R$ and its corresponding orientation of $T$}
    \label{fig:orient}
\end{figure}
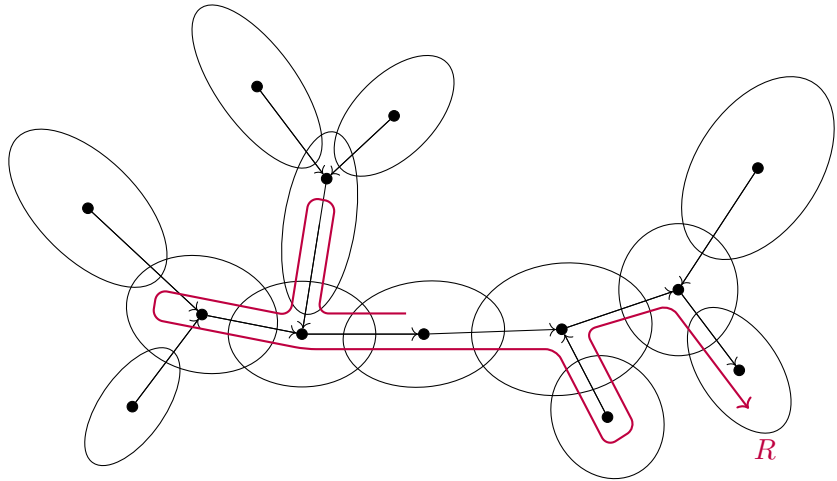

Tree-decompositions of finite adhesion have been used to study the structure of infinite graphs and their ends in e.g.\ \cites{bowler2013infinite, StarComb1StarsAndCombs,burger2022duality,carmesin2014all, carmesin2017topological,diestel1992end,pitz_shortCarmesin,thomas1989well}. Of course, some tree-decompositions of finite adhesion carry more information about the ends than others. For one, information content may be measured in terms of injectivity of $f_\mathcal{T}$. Indeed, a tree-decomposition consisting of a single part contains zero information, whereas a tree-decomposition $\mathcal{T}$ of finite adhesion that \emph{distinguishes} all the ends, i.e.\ where $f_\mathcal{T}$ is injective, contains more information about the end space -- although it may still give false hints, as for example ends of $T$ may not represent real ends of $G$. So even better would be a bijective $f_\mathcal{T}$, in which case we say that $\mathcal{T}$ \emph{represents} the ends of  $G$. 
On the other hand, while the trivial tree-decomposition into a single part always exist, some graphs $G$, such as the binary tree with one dominating vertex added to every rooted ray (cf.~Section~\ref{sec_applications}), are too complex to be distinguished or represented by a tree-decomposition of finite adhesion. Our first main result characterises precisely when these best-case scenarios occur; as a surprising by-product, we obtain that whenever a space $|G|$ can be distinguished by a tree-decomposition of finite adhesion, then it can also be represented. In fact, an even weaker condition suffices: As long as there is some tree-decomposition of finite adhesion into   
\emph{${\leq} 1$-ended parts}, i.e.\ a tree-decomposition such that at most one end is mapped to any given part under $f_\script{T}$, we also get a tree-decomposition representing $|G|$. 

Let's call a set of vertices $U \subseteq V(G)$ \emph{slender} if its closure $\overline{U} \subseteq |G|$ is scattered of finite Cantor-Bendixson rank; in other words, if successively taking the Cantor-Bendixson derivative of its closure $\overline{U} \subseteq |G|$ yields the empty set after finitely many iterations, cf.~Section~\ref{sec_topnotions}.

With this notion, our first main result reads as follows.

\begin{thm}
\label{thm_main}
The following are equivalent for any connected graph $G$ with at least one end:
\begin{enumerate}
    \item\label{main_item1} There is a tree-decomposition of finite adhesion that represents $\Omega(G)$.
    \item\label{main_item2} There is a tree-decomposition of finite adhesion that distinguishes $\Omega(G)$.
    \item\label{main_item3} There is a tree-decomposition of finite adhesion into ${\leq} 1$-ended parts. 
    \item\label{main_item4} $V(G)$ is a countable union of slender sets.
\end{enumerate}
\end{thm}

It is clear that any assertion from (\ref{main_item1}) to (\ref{main_item3})  implies the next. The idea for $(\ref{main_item3}) \Rightarrow (\ref{main_item4})$ is that for any fixed integer $n$, the union over all parts within distance $n$ from the root is a slender set of vertices, and $V(G)$ clearly is a countable union of these sets. Thus, the main contribution behind Theorem~\ref{thm_main} is the implication $(\ref{main_item4}) \Rightarrow (\ref{main_item1})$, which employs recently developed techniques of \emph{envelopes} from \cites{kurkofkapitz_rep,pitz_shortCarmesin} and \emph{rayless normal trees} from \cite{kurkofka2021approximating}. The proof of Theorem~\ref{thm_main} is given in Section~\ref{sec_8}.

A slightly different way to measure information captured by some tree-decomposition of finite adhesion is motivated by the  observation that end spaces of trees are well-understood: They are precisely the completely ultra-metrizable spaces. This suggests preferring  tree-decompositions $\mathcal{T}$ where $f_\mathcal{T}$ sends as many ends to $\Omega(T)$ as possible. In this case, there is hope to understand the subset $\Psi=f^{-1}_\mathcal{T}[\Omega(T)] \subseteq \Omega(G)$ called the \emph{boundary} of the tree-decomposition, with the best case being that $\mathcal{T}$ \emph{\textnormal{[}homeomorphically\textnormal{]} displays} its boundary, meaning that $f_\mathcal{T}$ restricts to a bijection [homeomorphism] between $\Psi$ and $\Omega(T)$, cf.~Figures~\ref{fig:my_labelBLA} and ~\ref{fig:my_labelBLA2}.

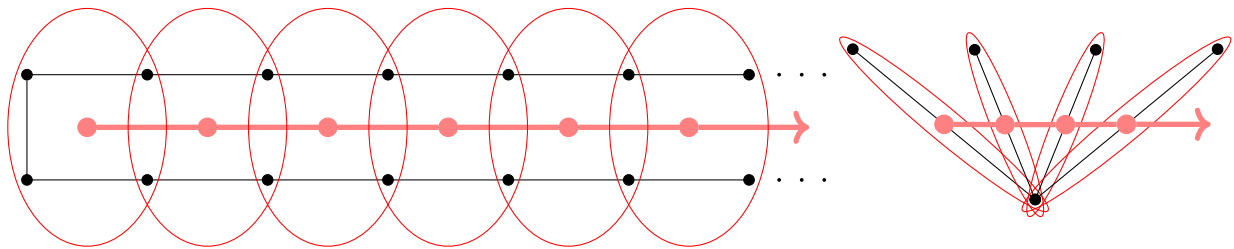
\begin{figure}[h]
    \centering
   \begin{tikzpicture}
    \tikzstyle{white node}=[draw,circle,fill=black,minimum size=4pt,inner sep=0pt]
    \tikzstyle{dot}=[draw,circle,fill=black,minimum size=1pt,inner sep=0pt]
    \tikzstyle{red node}=[draw,circle,color=red!50,fill=red!50,minimum size=7pt,inner sep=0pt]
    \tikzstyle{red circle}=[draw,circle,color=red,minimum size=75pt,inner sep=0pt]
    \tikzstyle{red edge} = [draw,line width=2pt,-,red!50]
    \tikzstyle{red dot}=[draw,circle,color=red,fill=red,minimum size=2pt,inner sep=0pt]
                                    
        \draw (0,-.1) node[white node] (v) {}
        -- ++(0:1.6cm) node[white node] () {}
        -- ++(0:1.6cm) node[white node] () {}
        -- ++(0:1.6cm) node[white node] () {}
        -- ++(0:1.6cm) node[white node] () {}
        -- ++(0:1.6cm) node[white node] () {}
        -- ++(0:1.6cm) node[white node] () {};
        
        \draw (v)
        -- ++(-90:1.4cm) node[white node] ()  {}
        -- ++(0:1.6cm) node[white node] () {}
        -- ++(0:1.6cm) node[white node] () {}
        -- ++(0:1.6cm) node[white node] () {}
        -- ++(0:1.6cm) node[white node] () {}
        -- ++(0:1.6cm) node[white node] () {}
        -- ++(0:1.6cm) node[white node] () {};
        
         \draw (0.8,-0.8) node[red node] (1) {}
        -- ++(0:1.6cm) node[red node] (2) {}
        -- ++(0:1.6cm) node[red node] (3) {}
        -- ++(0:1.6cm) node[red node] (4) {}
        -- ++(0:1.6cm) node[red node] (5) {}
        -- ++(0:1.6cm) node[red node] (6) {};

                \path[red edge] (1) -- (2);
                \path[red edge] (2) -- (3);
                \path[red edge] (3) -- (4);
                \path[red edge] (4) -- (5);
                \path[red edge] (5) -- (6);
                \draw[->,line width=2pt,red!50] (6) -- ++(0:1.6cm);

        \draw[color=red] (0.8,-0.8) ellipse (30pt and 45pt);
        \draw[color=red] (2.4,-0.8) ellipse (30pt and 45pt);
        \draw[color=red] (4,-0.8) ellipse (30pt and 45pt);
        \draw[color=red] (5.6,-0.8) ellipse (30pt and 45pt);
        \draw[color=red] (7.2,-0.8) ellipse (30pt and 45pt);
        \draw[color=red] (8.8,-0.8) ellipse (30pt and 45pt);

        \draw (10,-.1) node[dot] () {};
        \draw (10.3,-.1) node[dot] () {};
        \draw (10.6,-.1) node[dot] () {};
        
        \draw (10,-1.5) node[dot] () {};
        \draw (10.3,-1.5) node[dot] () {};
        \draw (10.6,-1.5) node[dot] () {};
 \end{tikzpicture}
\begin{tikzpicture}[xscale =0.7]
    \tikzstyle{white node}=[draw,circle,fill=black,minimum size=4pt,inner sep=0pt]
    \tikzstyle{dot}=[draw,circle,fill=black,minimum size=1pt,inner sep=0pt]
    \tikzstyle{red node}=[draw,circle,color=red!50,fill=red!50,minimum size=7pt,inner sep=0pt]
    \tikzstyle{red circle}=[draw,circle,color=red,minimum size=75pt,inner sep=0pt]
    \tikzstyle{red edge} = [draw,line width=2pt,-,red!50]
    \tikzstyle{red dot}=[draw,circle,color=red,fill=red,minimum size=2pt,inner sep=0pt]
          
                 \draw (0,0) node[white node] (v) {}
        -- ++(150:2cm) node[red node] (1)  {}
        -- ++(150:2cm) node[white node] ()  {};
 \draw[rotate=150,color=red] (2cm,0cm) ellipse (65pt and 7pt); 
 
         \draw (0,0) node[white node] (v) {}
        -- ++(120:1.15cm) node[red node] (2)  {}
        -- ++(120:1.15cm) node[white node] () {};
 \draw[rotate=120,color=red] (1.15cm,0cm) ellipse (40pt and 7pt); 
      
        \draw (v)
        -- ++(60:1.15cm) node[red node] (3)  {}
        -- ++(60:1.15cm) node[white node] ()  {};
\draw[rotate=60,color=red] (1.15cm,0cm) ellipse (40pt and 7pt);

        \draw (v)
        -- ++(30:2cm) node[red node] (4)  {}
        -- ++(30:2cm) node[white node] ()  {};
\draw[rotate=30,color=red] (2cm,0cm) ellipse (65pt and 7pt);

                \path[red edge] (1) -- (2);
                \path[red edge] (2) -- (3);
                \path[red edge] (3) -- (4);
                \draw[->,line width=2pt,red!50] (4) -- ++(0:1.6cm);
      \draw (0,-.5) node (x) {};
\end{tikzpicture}
\caption{Examples of tree-decompositions (in red) of graphs (in black) failing to display their boundaries.}
    \label{fig:my_labelBLA}
\end{figure}

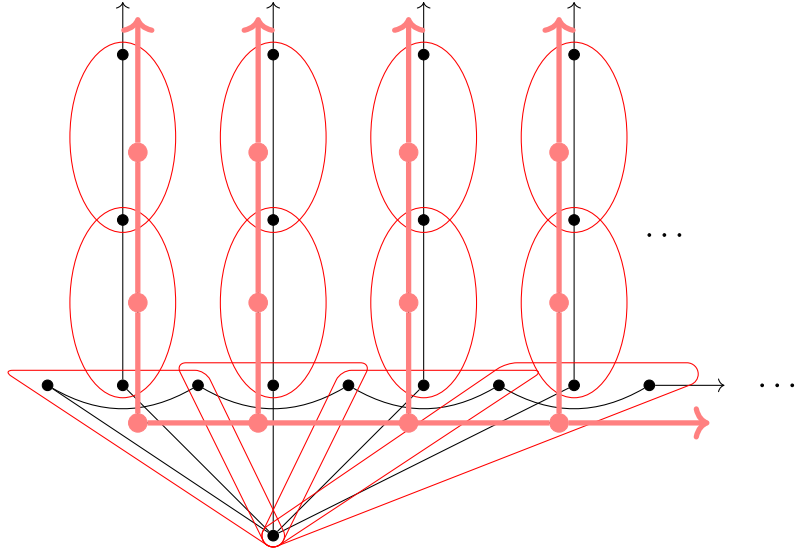
\begin{figure}
\begin{tikzpicture}
    \tikzstyle{white node}=[draw,circle,fill=black,minimum size=4pt,inner sep=0pt]
    \tikzstyle{dot}=[draw,circle,fill=black,minimum size=0pt,inner sep=0pt]
    \tikzstyle{p}=[draw,circle,fill=black,minimum size=1pt,inner sep=0pt]
    \tikzstyle{red node}=[draw,circle,color=red!50,fill=red!50,minimum size=7pt,inner sep=0pt]
    \tikzstyle{red circle}=[draw,circle,color=red,minimum size=75pt,inner sep=0pt]
    \tikzstyle{red edge} = [draw,line width=2pt,-,red!50]
    \tikzstyle{edge} = [draw,line width=0.2pt,-]
    \tikzstyle{red dot}=[draw,circle,color=red,fill=red,minimum size=2pt,inner sep=0pt]
      \draw (0,0) node[white node] (c) {};
       \draw (-3,2) node[white node] (1) {};
     \draw (-2,2) node[white node] (x1) {};
     \draw (-1,2) node[white node] (2) {};
     \draw (0,2) node[white node] (y1) {};
     \draw (1,2) node[white node] (3) {};
     \draw (2,2) node[white node] (z1) {};
     \draw (3,2) node[white node] (4) {};
     \draw (4,2) node[white node] (q1) {};
     \draw (5,2) node[white node] (5) {};
     
       \draw (c) -- (1);
     \draw (1) to [out=-30,in=-150] (2);
 \draw (2) to [out=-30,in=-150] (3);
 \draw (3) to [out=-30,in=-150] (4);
 \draw (4) to [out=-30,in=-150] (5); 
\draw[->,out=-30,in=-150] (5) -- ++(0:1cm);  

       \draw (c) -- (x1)
        -- ++(90:1.1cm) node[dot] (a1)  {}
        -- ++(90:1.1cm) node[white node] () {}
        -- ++(90:1.1cm) node[dot] (a2)  {}
        -- ++(90:1.1cm) node[white node] (a3) {};
        \draw[->] (a3) -- ++(90:0.7cm);

        \draw (c) -- (y1)
        -- ++(90:1.1cm) node[dot] (b1)  {}
        -- ++(90:1.1cm) node[white node] ()  {}
        -- ++(90:1.1cm) node[dot] (b2)  {}
        -- ++(90:1.1cm) node[white node] (b3)  {};
        \draw[->] (b3) -- ++(90:0.7cm);

               \draw (c) -- (z1)
        -- ++(90:1.1cm) node[dot] (c1)  {}
        -- ++(90:1.1cm) node[white node] ()  {}
        -- ++(90:1.1cm) node[dot] (c2)  {}
        -- ++(90:1.1cm) node[white node] (c3)  {};
        \draw[->] (c3) -- ++(90:0.7cm);
        
            \draw (c) -- (q1)
        -- ++(90:1.1cm) node[dot] (d1)  {}
        -- ++(90:1.1cm) node[white node] ()  {}
        -- ++(90:1.1cm) node[dot] (d2)  {}
        -- ++(90:1.1cm) node[white node] (d3)  {};
        \draw[->] (d3) -- ++(90:0.7cm);
 
\draw[rotate=90,color=red] (c1) ellipse (36pt and 20pt);
\draw[rotate=90,color=red] (c2) ellipse (36pt and 20pt);
\draw[rotate=90,color=red] (b1) ellipse (36pt and 20pt);
\draw[rotate=90,color=red] (b2) ellipse (36pt and 20pt);
\draw[rotate=90,color=red] (a1) ellipse (36pt and 20pt);
\draw[rotate=90,color=red] (a2) ellipse (36pt and 20pt);
\draw[rotate=90,color=red] (d1) ellipse (36pt and 20pt);
\draw[rotate=90,color=red] (d2) ellipse (36pt and 20pt);

\draw [red] (20:0.15cm) arc[start angle=20, end angle=-125,radius=0.15cm] [rounded corners] --  (-3.6,2.2)  -- (-0.9,2.2) [sharp corners]  -- cycle;

\draw [red] (-30:0.15cm) arc[start angle=-30, end angle=-150,radius=0.15cm] [rounded corners] -- (-1.3,2.3) -- (1.3,2.3) [sharp corners] -- cycle;

\draw [red] (-55:0.15cm) arc[start angle=-55, end angle=-200,radius=0.15cm] [rounded corners] -- (0.9,2.2) -- (3.6,2.2) [sharp corners] -- cycle;

\draw [red] (-60:0.15cm) arc[start angle=-60, end angle=-240,radius=0.15cm] [rounded corners] -- (3.1,2.3) [sharp corners] -- (5.5,2.3) arc[start angle=90, end angle=-70,radius=0.15cm]  --  cycle;

       \draw (-1.8,1.5) node[red node] (r1) {}
        -- (-0.2,1.5) node[red node] (r2) {}
        --  (1.8,1.5) node[red node] (r3) {}
        --  (3.8,1.5) node[red node] (r4) {};
 \draw[->,line width=2pt,red!50] (r4) -- ++(00:2cm);
 
 \draw (r1)
       -- ++(90:1.6cm) node[red node] (r11)  {}
       -- ++(90:2cm) node[red node] (r12)  {};
 
 \draw[->,line width=2pt,red!50] (r12) -- ++(90:1.8cm);       
\draw (r2)
        -- ++(90:1.6cm) node[red node] (r21)  {}
        -- ++(90:2cm) node[red node] (r22)  {};
       \draw[->,line width=2pt,red!50] (r22) -- ++(90:1.8cm);  
\draw (r3)
 -- ++(90:1.6cm) node[red node] (r31)  {}
        -- ++(90:2cm) node[red node] (r32)  {};
       \draw[->,line width=2pt,red!50] (r32) -- ++(90:1.8cm); 
       
       \draw (r4)
 -- ++(90:1.6cm) node[red node] (r41)  {}
        -- ++(90:2cm) node[red node] (r42)  {};
       \draw[->,line width=2pt,red!50] (r42) -- ++(90:1.8cm);

              \path[red edge] (r1) -- (r4);
                \path[red edge] (r1) -- (r11);
                \path[red edge] (r11) -- (r12);
               \path[red edge] (r2) -- (r21);
              \path[red edge] (r21) -- (r22);
      \path[red edge] (r3) -- (r31);
      \path[red edge] (r31) -- (r32);
            \path[red edge] (r4) -- (r41);
      \path[red edge] (r41) -- (r42);

      \draw (5,4) node[p] () {};
        \draw (5.2,4) node[p] () {};
        \draw (5.4,4) node[p] () {};
        
         \draw (6.5,2) node[p] () {};
        \draw (6.7,2) node[p] () {};
        \draw (6.9,2) node[p] () {};
\end{tikzpicture}

\caption{Example of a tree-decomposition (in red) that displays all ends of a countable star of rays (in black) but fails to display them homeomorphically.}
    \label{fig:my_labelBLA2}
\end{figure}

At first glance, however, it does not seem useful at all when $f_\script{T}$ maps all ends of $G$ into $\Omega(T)$ but the function is very much non-injective. 
However, this information is enough to guarantee a normal spanning tree, from which the space $|G|$ is easily understood. Indeed, given previous work in the field due to Jung and Diestel \cites{jung1969wurzelbaume,diestel2006end,diestel1994depth}, it is not hard to verify that the following assertions are equivalent, see Theorem~\ref{thm:display_sets} for details:

\begin{itemize}
     \item There is a tree-decomposition of finite adhesion that (homeomorphically) displays $\Omega(G)$.
    \item There is a tree-decomposition of finite adhesion with boundary $\Omega(G)$.
     \item $|G|$ is (completely) metrizable.
        \item $V(G)$ is a countable union of closed sets in $|G|$.
    \item $G$ has a normal spanning tree.
\end{itemize}

Now our second main result provides a local version of the above equivalences, characterising precisely which subsets $\Psi$ of $\Omega(G)$ can be (homeomorphically) displayed. 
Indeed, a striking, recent result by Carmesin \cite{carmesin2014all} says that it is always possible to display the set of undominated ends of a graph $G$. In \cite{burger2022duality}, Bürger and Kurkofka partially localized Carmesin's result by constructing tree-decompositions of finite adhesion (with additional desirable properties) that display the boundary $\partial U$ of prescribed infinite sets of vertices $U \subseteq V(G)$ where none of the ends in $\partial U$ are dominated.
Carmesin also asked for a characterisation of those pairs of a graph $G$ and a subset $\Psi \subseteq \Omega(G)$ for which $G$ has a tree-decomposition displaying $\Psi$ \cite[p.~549]{carmesin2014all}. This problem has also been reiterated in \cite[Problem~3.22]{StarComb1StarsAndCombs}. Theorem~\ref{thm_main3} below answers this question. 

Another set of questions in infinite topological graph theory concerns so-called $\Psi$-graphs $|G|_\Psi$, i.e.\ subspaces of $|G|$ of the form $|G|_\Psi=G\cup \Psi \subseteq |G|$ for a set of ends $\Psi \subseteq \Omega(G)$. $\Psi$-graphs have been studied in connection with infinite matroids  \cites{bowler2013infinite,bowler2013ubiquity,diestel2017dual}: For example, the topological circles (copies of the unit circle $S^1$) in $|G|_\Psi$ form the cycles of an infinite matroid whenever $\Psi$ belongs to the Borel $\sigma$-algebra of $\Omega(G)$ \cite{bowler2013infinite}.

It turns out that the correct generalisation of the 3rd bullet above about metrizability of $|G|$ involves precisely the property of complete metrizability of $\Psi$-spaces.

\begin{thm}\label{thm_main3}
For any connected graph $G$ and a set $\Psi$ of ends of $G$ the following are equivalent:
\begin{enumerate}
  \item There is a tree-decomposition of finite adhesion homeomorphically displaying $\Psi$.
    \item There is a tree-decomposition of finite adhesion displaying $\Psi$.
    \item There is a tree-decomposition of finite adhesion with boundary $\Psi$.
    \item $|G|_\Psi$ is completely metrizable.
    \item\label{itemGdelta} $\Psi$ is $G_\delta$ in $|G|$.
\end{enumerate}
\end{thm}

Note that from Theorem~\ref{thm_main3} one easily reobtains the above equivalences in the case $\Psi=\Omega$.  Indeed, only item (\ref{itemGdelta}) needs to be commented on: For this, note that saying that  $\Psi = \Omega$ is $G_\delta$ in $|G|$ means $\Psi=\Omega$ is a countable intersection of open sets, which turns out to be equivalent to $V(G)$ being a countable union of closed sets in $|G|$. Also note that $\Psi$ being a $G_\delta$ means that $\Psi$ is a fairly simple element of the Borel $\sigma$-algebra on $|G|$, and in fact, using Theorem~\ref{thm_main3} it is not hard to establish that $|G|_\Psi$ gives an infinite matroid in the special case from \cite{bowler2013infinite} where $\Psi \subseteq |G|$ is $G_\delta$. 

Carmesin's result that the undominated ends $\Psi$ of any connected graph can always be displayed now follows easily from Theorem~\ref{thm_main3}: Simply note that fixing any vertex $v$ and considering the set $B_n(v)$ of all vertices in $G$ within graph distance at most $n$ from $v$, the set  $\Psi$ is the intersection of the countably  many open sets $O_n=|G| \setminus \overline{B_n(v)}$ (for $n \in \NN)$ and hence $G_\delta$, see Theorem~\ref{thm_carmensin_Gdelta}. 

Furthermore, Theorem~\ref{thm_main3} also provides tree-decompositions that (homeomorphically) display the undominated ends in the boundary $\partial U$ of any fixed infinite set of vertices $U \subseteq V(G)$, strengthening  the above mentioned result by Bürger and Kurkofka from \cite{burger2022duality}; see Theorem~\ref{thm_carljan_Gdelta}.

A number of natural questions remain on the topic which subsets of ends can be distinguished. 
\begin{problem}
Characterise which $\Psi \subseteq \Omega(G)$ can be distinguished.
\end{problem}
Given two distinct ends $\omega_1, \omega_2$ of a graph $G$ write $n(\omega_1,\omega_2) \in \NN$ for the minimal order of a separation in $G$ that is oriented differently by  $\omega_1 $ and $ \omega_2$.
We say that a tree-decomposition $\mathcal{T}$ with decomposition tree $T$ \emph{efficiently distinguishes} a set of ends $\Psi$ if $\mathcal{T}$ distinguishes $\Psi$ with the additional property that for each $\psi_1 \neq \psi_2 \in \Psi$ there is an edge $e$ on the path in $T$ between $f_\mathcal{T}(\psi_1)$ and $f_\mathcal{T}(\psi_2)$ with $|X_e| = n(\omega_1,\omega_2)$.
\begin{problem}
Characterise which $\Psi \subseteq \Omega(G)$ can be efficiently distinguished.
\end{problem}
An end $\omega$ of a graph \emph{thin} if all families of disjoint $\omega$-rays are finite, and \emph{thick} otherwise.
Our next problem extends a problem of Diestel \cite{diestel1992end}, asking for which graphs there is a tree-decomposition of finite adhesion displaying precisely its thin ends.
Carmesin \cite{carmesin2014all} constructed a graph for which there is no such tree-decomposition, and we construct a different counterexample in Example \ref{thin_ends} with help of our characterisation of displayable sets of ends from Theorem \ref{thm_main3}.
We propose a different way in which a tree-decomposition of finite adhesion might distinguish the thin ends from the thick ends and ask which other bipartitions of $\Omega(G)$ can be distinguished in the same way:
\begin{problem}
Characterise for which bipartitions $\Omega(G) = \Omega_1 \sqcup \Omega_2$ there is a tree-decompositions $\mathcal{T}$ of finite adhesion with $f_\mathcal{T}(\Omega_1) \cap f_\mathcal{T}(\Omega_2)= \emptyset$.
\end{problem}

We conclude with two problems concerning metrizability in end spaces.
\begin{problem}
\label{prob16}
Characterise which subspaces $\Psi \subseteq \Omega(G)$ are metrizable or completely metrizable.
\end{problem}

\begin{problem}
Characterise which spaces $|G|_\Psi$ are metrizable.
\end{problem}
For $\Psi = \Omega(G)$, an answer to Problem~\ref{prob16} is given in  \cite{kurkofka2021approximating}.

\section{Preliminaries}\label{sec_prelims}

\subsection{Ends} 

For graph theoretic terms we follow the terminology in \cite{diestel2015book}, and in particular \cite[Chapter~8]{diestel2015book} for ends in graphs. A $1$-way infinite path is called a \emph{ray} and the subrays of a ray are its \emph{tails}. Two rays in a graph $G = (V,E)$ are \emph{equivalent} if no finite set of vertices separates them; the corresponding equivalence classes of rays are the \emph{ends} of $G$. If $\omega$ is an end of $G$ and $R\in\omega$, we call $R$ an $\omega$-ray. The set of ends of a graph $G$ is denoted by $\Omega = \Omega(G)$.

The \emph{degree} $\deg(\omega)$
of an end $\omega$ is the supremum of the sizes of collections of pairwise disjoint rays in $\omega$; Halin showed that this supremum is always attained, see~\cite[Theorem~8.2.5]{diestel2015book}. Ends are called \emph{thin} if they have finite degree, and \emph{thick} otherwise.

We say that a vertex $v$ \emph{dominates} a ray $R$ if there is a subdivided star with centre $v$ and leaves in $R$. Whenever two rays $R$ and $R'$ are equivalent, a vertex dominates $R$ if and only if it dominates $R'$. Thus we can say that a vertex \emph{dominates} an end $\omega$ if and only if it dominates one ray (and thus all rays) from $\omega$. In this case, we say $\omega$ is \emph{dominated}. 

\subsection{Ends and directions}
If $X \subseteq V$ is finite and $\omega \in \Omega$, 
there is a unique component of $G-X$ that contains a tail of every $\omega$-ray. We denote this component by $C(X,\omega)=C_G(X,\omega)$ and say that $\omega$ \emph{lives in $C(X,\omega)$}.

A \emph{direction} on $G$ is a function $d$ that assigns to every finite $X \subseteq V$ one of the components of $G-X$ so that $d(X) \supseteq d(X')$ whenever $X \subseteq X'$. For every end $\omega$, the map $X \mapsto C(X,\omega)$ is easily seen to be a direction. Conversely, every direction is defined by an end in this way:

\begin{thm}[Diestel \& K\"uhn {\cite{diestel2003graph}}]
\label{thm_direction}
For every direction $d$ on a graph $G$ there is an end $\omega$ such that $d(X) = C(X, \omega)$ for every finite $X \subseteq V(G)$.
\end{thm}

\subsection{Star-Comb Lemma}

Given a set of vertices $U$, a \emph{comb attached to $U$} consists of a ray $R$ together with infinitely many disjoint $R$--$U$ paths (possibly trivial). A \emph{star attached to $U$} is a  subdivided infinite star with all leaves in $U$.

\begin{lemma}[Star-Comb Lemma {\cite[Lemma~8.2.2]{diestel2015book}}]
\label{lem_starcomb}
Let $U$ be an infinite set of vertices in a connected graph $G$. Then $G$ contains a star or a comb attached to $U$.
\end{lemma}

\subsection{End spaces}\label{section:end_spaces}
If $X \subseteq V$ is a finite set of vertices and $\omega \in \Omega$, 
then $\Omega(X,\omega) = \Omega_G(X,\omega) =  \set{\phi \in \Omega}:{C(X,\phi) = C(X,\omega)}$ denotes the set of all ends that live in $C(X,\omega)$. We put $\hat{C}(X,\omega) = C(X,\omega) \cup \Omega(X,C)$.

The collection of singletons $\{v\}$ for $v \in V(G)$ together with all sets of the form $\hat{C}(X,\omega)$ for finite $X\subseteq V$ and $\omega \in \Omega(G)$ forms a basis for a topology on $V(G) \cup \Omega(G)$.
This topology is Hausdorff, and it is \emph{zero-dimensional} in the sense that it has a basis consisting of closed-and-open sets. 

Given a set of vertices $U \subseteq V(G)$, we write $\partial U$ for its boundary, i.e.\ the set of ends in $\overline{U}$. It is well-known that $\omega \in \partial U$ if and only if there is a comb attached to $U$ with spine in $\omega$.

If $H$ is a subgraph of $G$, then rays equivalent in $H$ remain equivalent in $G$; in other words, every end of $H$ can be interpreted as a subset of an end of $G$, so the natural inclusion map $\iota \colon \Omega(H) \to \Omega(G)$ is well-defined. A subgraph $H \subseteq G$ is \emph{end-faithful} if this inclusion map $\iota$ is a bijection from $\Omega(H)$ onto $\partial H \subseteq \Omega(G)$.\footnote{In the literature, the term end-faithful subgraph is sometimes used only for subgraphs $H \subseteq G$ with $\partial H = \Omega(G)$.}

We now describe one common way to extend this topology on $V(G) \cup \Omega(G)$ to a topology on $|G| = G \cup \Omega(G)$, the graph $G$ together with its ends. 
This topology, called \textsc{MTop}, has a basis formed by all open sets of $G$ considered as a metric length-space (i.e.\ every edge together with its endvertices forms a unit interval of length~$1$, and the distance between two points of the graph is the length of a shortest arc in $G$ between them), together with basic open neighbourhoods for ends of the form
\begin{align*}
    \hat{C}_\varepsilon(X,\omega) := \hat{C}(X,\omega)  \cup \mathring{E}_\varepsilon(X, C(X,\omega)),
\end{align*}
where $\mathring{E}_\varepsilon(X, C(X,\omega))$ denotes the open ball around $C(X,\omega)$ in $G$ of radius $\varepsilon < 1$.

\subsection{Tree orders and normal trees}
The \emph{tree order} of a tree $T$ with root $r$ is a partial order on $V(T)$ which is defined by setting $u \leq v$ if $u$ lies on the unique path $rTv$ from $r$ to $v$ in $T$. Given $n \in \mathbb{N}$, the \emph{$n$th level} $T^n$ of $T$ is the set of vertices at distance $n$ from $r$ in $T$, and by $T^{\leq n}$ we denote the union over the first $n$ levels. The \emph{down-closure} of a vertex $v$ is the set $\lceil v \rceil := \{\,u \colon u \leq v\,\}$; its \emph{up-closure} is the set $\lfloor v \rfloor := \{\,w \colon v \leq w\,\}$. The down-closure of $v$ is always a finite chain, the vertex set of the path $rTv$. A ray $R \subseteq T$ starting at the root is called a \emph{normal ray} of $T$.

A rooted spanning tree $T$ of a graph $G$ is \emph{normal} in~$G$ if the endvertices of every edge of $G$ are comparable in the tree order of~$T$. Normal spanning trees are always end-faithful \cite[Lemma~8.2.3]{diestel2015book}.

A rooted, not necessarily spanning, tree $T$ contained in a graph $G$ is \emph{normal} in~$G$ if the endvertices of every \mbox{$T$-path} in $G$ are comparable in the tree-order of~$T$. 
Here, for a given subgraph $H \subseteq G$, a path $P$ in $G$ is said to be an $H$-\emph{path} if $P$ is non-trivial and meets $H$ exactly in its endvertices. Clearly, if $T$ is spanning, this reduces to the earlier condition, as in this case all $T$-paths are chords.
We remark that for a normal tree $T \subseteq G$ the neighbourhood $N(D)$ of every component $D$ of $G-T$ forms a chain in $T$. 
The following result can be found in \cite{kurkofka2021approximating}.
\begin{thm}
\label{thm:RNT}
Let $G$ be a connected graph. For every open cover $\cO$ of $\Omega(G)$, there is a rayless normal tree $T$ in $G$ such that for every component $C$ of $G-T$ there is a set $O\in\cO$ such that $\partial C\subseteq O$.
\end{thm}

Additionally, we need the following easy lemma: 
We say a set of vertices $U$ in a graph $G$ has \emph{finite adhesion}, if and only if every component of $G-U$ has a finite neighbourhood in $U$.

\begin{lemma}\label{lem:extend_RNT}
Let $G$ be a connected graph and $T$ a rayless normal tree in $G$. Then $T$ has finite adhesion in $G$. Moreover, for every finite set $U\subseteq V(G)$ there is a rayless normal tree $T^* \supseteq T$ in $G$ such that $U\subseteq V(T^*)$.
\end{lemma}

\begin{proof}
For the proof that $T$ has finite adhesion in $G$, let $C$ be any component of $G-T$. Since $T$ is normal, the neighbourhood of $C$ is a chain in the tree order of $T$, and this chain is finite because $T^*$ is rayless.

Next, let $T^*$ be a rayless normal tree in $G$ extending the tree $T$ which contains maximally many vertices from $U$. We show that $T^*$ contains all vertices from $U$. Suppose for a contradiction that there is a vertex $u\in U$ with $u\notin V(T^*)$ and let $C$ be the component of $G-T^*$ containing $u$. We showed in the first paragraph of this proof that the neighbourhood $N(C)$ of $C$ is a finite chain in the tree order of $T^*$. Let $v$ be its maximal element and $v'$ a neighbour of $v$ in $C$. Then the union of $T^*$ with the edge $vv'$ and a $v'$--$u$ path in $C$ is again a rayless normal tree with $T$ as a subgraph, contradicting the maximality of $T$.
\end{proof}

\subsection{Topological notions}
\label{sec_topnotions}
A subspace $Y$ of a topological space $X$ is \emph{discrete} if every singleton of $Y$ is open in the subspace topology.

A \emph{$G_\delta$-set} of a topological space $X$ is a countable intersection of open sets. An \emph{$F_\sigma$-set} is a countable union of closed sets. Note that the complement of a $G_\delta$-set is always a $F_\sigma$-set and vice versa.

\begin{lemma}
\label{lem_Fsigma}
Let $G$ be a graph and $\Psi\subseteq\Omega(G)$. Then $V(G)\cup\Psi$ is $F_\sigma$ in $|G|$ if and only if $G\cup\Psi$ is $F_\sigma$ in $|G|$.
\end{lemma}

\begin{proof}
 The backwards direction follows from that fact that $V \cup \Omega$ is closed in $|G|$, so $V\cup\Psi$ is closed in $G\cup\Psi$, and closed subsets of $F_\sigma$-sets are themselves $F_\sigma$.
 
 Conversely, assume $V \cup \Psi = \bigcup_{n \in \NN} X_n$ is a countable union of closed sets $X_n$ of vertices and ends  in $|G|$. Without loss of generality, we have $X_{n} \subseteq X_{n+1}$. Let $V_n = X_n \cap V(G)$. Then $\bigcup_{n \in \NN} V_n = V(G)$ and $G = \bigcup_{n \in \NN} G[V_n]$. But then the induced subsets $G[X_n] := G[V_n] \cup X_n$ are also closed in $|G|$, and so $G \cup \Psi = \bigcup_{n \in \NN} G[X_n]$ is $F_\sigma$ in $|G|$, too. 
\end{proof}

A set of vertices $U$ in a graph $G$ is \emph{dispersed} if it can be separated from any ray in $G$ by a finite set of vertices. This is equivalent to the property of $U$ being closed in $|G|$.

Given a topological space $X$, the \emph{derived space} $X'$ is the subspace $X' \subseteq X$ obtained by removing all isolated points from $X$.
By transfinite induction, one defines a decreasing sequence of subsets of $X$ by setting 
$X^{(0)} = X$, $X^{(\alpha+1)} =( X^{(\alpha)})'$ in the successor case, and $X^{(\lambda)} = \bigcap_{\alpha < \lambda} X^{(\alpha)}$ for all limit ordinals $\lambda$. 
For cardinality reasons, this transfinite sequence must eventually be constant. The smallest ordinal $\alpha$ such that $X^{(\alpha+1)} = X^{(\alpha)}$ is called the \emph{Cantor–Bendixson rank} of $X$. If $X^{(\alpha)} = \emptyset$ for some ordinal $\alpha$, then $X$ is \emph{scattered}. 

A set of vertices $U \subseteq V(G)$ is \emph{slender} if $X = \overline{U} \subseteq |G|$ satisfies $X^{(n)} = \emptyset$ for some $n \in \NN$.

\section{Tree-decompositions}

A [rooted] \emph{tree-decomposition} of a graph $G$ is a pair $\script{T} = \p{T,\script{V}}$ where $T$ is a [rooted] tree and $\cV = (V_t \colon t \in T)$ is a family of vertex sets of $G$ called \emph{parts} such that the following holds (see also \cite[\S12.3]{diestel2015book}):
\begin{enumerate}[label=(T\arabic*)]
\item for every vertex $v$ of $G$ there exists $t \in T$ such that $v \in V_t$;
\item for every edge $e$ of $G$ there exists $t \in T$ such that $e \in G[V_t]$; and
\item $V_{t_1} \cap V_{t_3} \subseteq V_{t_2}$ whenever $t_2$ lies on the $t_1$--$t_3$ path in $T$.
\end{enumerate}
Let $e = xy$ be any edge of $T$ and let $T_x$ and $T_y$ be the two components of $T-e$ with $x \in T_x$ and $y \in T_y$. Each edge $e=xy$ of $T$ in a tree-decomposition gives rise to a separator $X_e:=V_x \cap V_y$ called the separator \emph{induced by} the edge $e$, which separates $A_x=\bigcup_{t \in T_x} V_t$ from $A_y=\bigcup_{t \in T_y} V_t$.
The tree-decomposition has \emph{finite adhesion} if all separators of $G$ induced by the edges of $T$ are finite. 

Given a tree-decomposition $\script{T} = \p{T,\script{V}}$ of finite adhesion of $G$, any end $\omega$ of $G$ orients each edge $e=xy$ of $T$ according to whether $\omega$ lives in a component of $G[A_x] - X_e$ or $G[A_y]-X_e$. This orientation of $T$ points towards a node of $T$ or to an end of $T$, and $\omega$ \emph{lives} in that part for that node or \emph{corresponds} to that end, respectively.

Let $f_\script{T} \colon \Omega(G) \to V(T) \cup \Omega(T)$ be the function mapping every end of $G$ to the node or end of $T$ that it lives in or corresponds to, respectively. We say that $\script{T}$ \emph{distinguishes} the ends of $G$ if $f_\script{T}$ is injective, and it \emph{represents} the ends of $G$ if $f_\script{T}$ is bijective.

We call $f_\script{T}^{-1}[\Omega(T)]$ the \emph{boundary} of $\script{T}$, and  $f_\script{T}^{-1}[V(T)]$ the \emph{interior} of $\script{T}$. We say that $\script{T}$ \emph{displays} a subset $\Psi\subseteq\Omega(G)$ if $\Psi$ is the boundary of $\script{T}$ and $f_\script{T}\upharpoonright\Psi \to \Omega(T)$ is bijective, and it \emph{homeomorphically displays} $\Psi$ if $f_\script{T}\upharpoonright\Psi \to \Omega(T)$ is a homeomorphism. We say that $\script{T}$ [\emph{bijectively}] \emph{distributes} a subset $\Xi\subseteq\Omega(G)$ if $\Xi$ is the interior of $\script{T}$ and $f_\script{T}\upharpoonright\Xi \to V(T)$ is injective [bijective]. 
Finally, we say that $\script{T}$ \emph{realises} [\emph{represents}] a partition $\Omega(G) = \Xi \sqcup \Psi$ of the end space of $G$ if $\script{T}$ [bijectively] distributes $\Xi$ and displays $\Psi$.

We conclude this section with a sufficient condition for tree-decompositions to (homeomorphically) display their boundary. We say a rooted tree-decomposition $(T,\mathcal{V})$ is \emph{upwards connected} if for every edge $e=xy \in E(T)$ with $x<y$, the induced subgraph $H_e:=G[A_y \setminus A_x] = G[A_y] - X_e$ (with $A_x$, $A_y$ and $X_e$ as above) is non-empty and connected (or equivalently, $H_e$ is a component of $G-X_e$).

\begin{lemma}
\label{lem_upwardsconndisplay}
Every upwards connected rooted tree-decomposition $\script{T}=(T,\cV)$ of finite adhesion of a graph $G$ homeomorphically displays its boundary.
\end{lemma}

\begin{proof}
Let $\Psi$ be the boundary of $\script{T}$. We show that $f:=f_\script{T} \upharpoonright \Psi : \Psi \to \Omega(T)$ is a homeomorphism.

For the proof that $f$ is injective, let $\psi_1\neq\psi_2\in\Psi$ and let $R_i$ be the $f(\psi_i)$-ray in $T$ starting in the root of $T$ for $i=1,2$.
There is a finite vertex set $S\subseteq V(G)$ such that $\psi_1$ and $\psi_2$ live in different components of $G-S$.
By (T1) there is a finite subtree $T'$ of $T$ containing the root of $T$ such that $S\subseteq\bigcup_{t\in T'}V_t=:G'$.
We denote the unique $T'$--$(T\setminus T')$ edge in $R_i$ by $e_i$ for $i=1,2$.
Then $\psi_i$ lives in $H_{e_i}$ (as defined above) which is a component of $G-G'$ since $\cT$ is upwards connected.
Since $\psi_1$ and $\psi_2$ live in different components of $G-S$, they also live in different components of $G-G'$.
It follows that $H_{e_1}\neq\ H_{e_2}$.
Therefore $e_1\neq e_2$, $R_1\neq R_2$, and thus $f(\psi_1)\neq f(\psi_2)$.

Next, for the proof that $f$ is onto, for each end $\omega$ of $T$ we find an end $\psi\in\Psi$ such that $f_\script{T}(\psi)=\omega$.
Let $R=re_0v_1e_1v_2e_2\ldots$ be the $\omega$-ray in $T$ starting in the root of $T$.
We have $\bigcap_{i\in\NN} H_{e_i}=\emptyset$ because each $H_{e_i}$ contains only vertices from parts $V_t$ such the distance of $t$ to the root of $T$ is greater than $i$.
In particular, for every finite subset $S$ of $V(G)$ there is a minimal integer $i$ such that $S\cap H_{e_i}=\emptyset$.
Since $H_{e_i}$ is connected and non-empty, there is a unique component $d(S)$ of $G-S$ with $H_{e_i}\subseteq d(S)$.
The function $d$ defines a direction on $G$ because the components $H_{e_0}\supseteq H_{e_1}\supseteq\ldots$ are nested and non-empty. 
By Theorem~\ref{thm_direction}, there is an end $\psi$ of $G$ such that $C_G(S,\psi)=d(S)$ for every finite subset $S\subseteq V(G)$.
In particular, we have $d(X_{e_i})=H_{e_i}$ for the separator $X_{e_i}$ corresponding to $e_i$, and hence $\psi$ lives in $H_{e_i}$ for all $i\in\NN$.
Consequently, $\psi$ lies in the boundary of $\script{T}$ and $f(\psi)=\omega$.

We now argue that $f$ is continuous (this part of the argument works for any tree-decomposition displaying $\Psi$ and doesn't yet require upwards connectedness). Indeed, let $\psi \in \Psi$ and $f(\psi) = \omega \in \Omega(T)$. For continuity, consider an arbitrary basic open neighbourhood $\Omega_T(T',\omega)$ of $\omega \in \Omega(T)$. Since $T$ is a tree, there is a unique $C_T(T',\omega)$--$T'$ edge $e=tt'$. Then $X_e = V_t \cap V_{t'}$ is finite since $\mathcal{T}$ had finite adhesion. Now $C_G(X_e,\psi)$ lies completely on one side of the separation $(A_e,B_e)$, and so all ends in $C_G(X_e,\psi)$ orient $e$ towards $\omega$, showing that $f[\Omega_G(X_e,\psi)] \cap \Psi \subseteq \Omega_T(T',\omega)$ as desired.

Finally, we show that $f^{-1}$ is continuous (this part fails without upwards connectedness, cf.~Figure~\ref{fig:my_labelBLA2}). Let $f^{-1}(\omega) = \psi \in \Psi$ as before and consider a  basic open neighbourhood $\Omega_G(S,\psi) \cap \Psi$ of $\psi \in \Psi$. Let $T'$ be a finite subtree of $T$ which contains the root of $T$ and such that $S\subseteq\bigcup_{t\in V(T')}V_t$. Since $\psi$ orients the unique $C_T(T',\omega)$--$T'$ edge $e$ towards $\omega$ and $H_e$ is connected, it follows that $H_e = C_G(X_e,\psi) \subseteq C_G(S,\psi)$. Thus, all ends that orient $e$ towards $\omega$ live in $C_G(S,\psi)$, giving 
$f^{-1}[\Omega_T(T',\omega)] \subseteq \Omega_{G}(S,\psi) \cap \Psi$ as desired.
\end{proof}

\section{Tree-decompositions displaying all ends}
In this section we answer the question which graphs have a tree-decomposition displaying all ends. It turns out that those are exactly the graphs with a normal spanning tree. A characterisation of those graphs by forbidden minors can be found in \cite{pitz2021proof}.

\begin{thm}
\label{thm_Psi=Omega}
The following are equivalent for any connected graph $G$:
\begin{enumerate}
    \item \label{item_Psi=Omega0} There is an upwards connected tree-decomposition of finite adhesion with connected parts that homeomorphically displays $\Omega(G)$.
    \item \label{item_Psi=Omegai} There is a tree-decomposition of finite adhesion displaying $\Omega(G)$.
    \item \label{item_Psi=Omegai'} There is a tree-decomposition of finite adhesion with boundary $\Omega(G)$.
     \item \label{item_Psi=Omegaii} $|G|$ is (completely)  metrizable.
    \item \label{item_Psi=Omegaiii} $\Omega(G)$ is $G_\delta$ in $|G|$.
    \item \label{item_Psi=Omegaiv} $G$ has a normal spanning tree.
\end{enumerate}
\end{thm}

\begin{proof}
The equivalence $(\ref{item_Psi=Omegaiii})\Leftrightarrow (\ref{item_Psi=Omegaiv})$ is a well-known result by Jung characterising the existence of normal spanning trees \cite{jung1969wurzelbaume}. In Jung's language, a connected graph has a normal spanning tree if and only if $V(G)$ is a countable union of dispersed sets; since dispersed sets are precisely the sets of vertices which are closed in $|G|$, this is equivalent to  $V=V(G)$ being $F_\sigma$ in $|G|$. 
By Lemma~\ref{lem_Fsigma}, this is equivalent to $G$ being $F_\sigma$ in $|G|$, which by taking complements is the same as $\Omega(G)$ being $G_\delta$ in $|G|$. 

The equivalence $(\ref{item_Psi=Omegaii})\Leftrightarrow (\ref{item_Psi=Omegaiv})$ is due to Diestel \cite{diestel2006end}.\footnote{For $(\ref{item_Psi=Omegaiv})\Rightarrow (\ref{item_Psi=Omegaii})$, Diestel only verifies that his metric is topologically compatible; but it not hard to see that his metric is in fact complete. See also Theorem~\ref{thm:display_sets}.}
The implications $(\ref{item_Psi=Omega0})\Rightarrow (\ref{item_Psi=Omegai})$ and $(\ref{item_Psi=Omegai})\Rightarrow (\ref{item_Psi=Omegai'})$ are trivial.

For $(\ref{item_Psi=Omegai'}) \Rightarrow (\ref{item_Psi=Omegaiii})$ suppose we have a tree-decomposition $(T,\mathcal{V})$ with root $r$ of finite adhesion with boundary $\Omega(G)$. 
We claim that $G[D_n]$ is closed, where
$$D_n:=\bigcup_{t \in T^{\leq n}} V_{t}.$$
Indeed, for any end $\omega$ of $G$ there is a unique ray $R=t_0t_1t_2\ldots$ starting at the root $t_0 = r$ corresponding to this end. Then $V_{t_n} \cap V_{t_{n+1}}$ is a finite separator that separates $D_n$ from the tails of all $\omega$-rays. Hence no end lives in the closure of $D_n$, so $G[D_n]$ is closed. It follows from property (T1) and (T2) of a tree-decomposition that $G = \bigcup_{n \in \NN} G[D_n]$ is $F_\sigma$, so by taking complements in $|G|$, we see that $\Omega(G)$ is $G_\delta$ in $|G|$.

Lastly, we show $(\ref{item_Psi=Omegaiv})\Rightarrow (\ref{item_Psi=Omega0})$. Something  similar  has been done in \cite{diestel1994depth}. Assume that $G$ has a normal spanning tree $T$ with root $r$. For every vertex $t$ of $T$, we define $V_t := \lceil t \rceil$ and show that $\cT:=(T,(V_t)_{t \in T})$ is a tree-decomposition of $G$ of finite adhesion that homeomorphically displays all its ends.
Since $T$ is normal, the end vertices of any edge $vw$ of $G$ are comparable in the tree order. If say $v < w$, then $e$ belongs to the part $V_w$ per definition, giving (T2). Further, if a vertex lies in two parts $V_v$ and $V_w$, it lies in $\lceil v \rceil \cap \lceil w \rceil$ and hence in all $V_t$ for vertices $t$ on the unique $v$--$w$ path in $T$. Thus we get property (T3), so we have a tree-decomposition.
It is clear that all parts are connected and since all parts are finite, also all adhesion sets are finite. Finally, $\cT$ is clearly upwards connected.
Therefore it follows from Lemma \ref{lem_upwardsconndisplay} that $\cT$ homeomorphically displays its boundary, which contains all ends of $G$ since all parts are finite.
\end{proof}

\section{Envelopes}

Let $G$ be a connected graph. An \emph{envelope} for a set of vertices $U \subseteq V(G)$ is a set of vertices $U^* \supseteq U$ of finite adhesion (i.e.\ such that every component of $G-U^*$ has only finitely many neighbours in $U^*$) with $\partial U^* = \partial U$. In \cite[Theorem~3.2]{kurkofkapitz_rep} it is proven that every set of vertices in a connected graph admits a connected envelope.

In the following, however, we need a stronger notion of an envelope that works for a set $X \subseteq V(G) \cup \Omega(G)$ of vertices and ends (and in particular, for a set $X$ consisting of ends only): An \emph{envelope} for such a set $X \subseteq V(G) \cup \Omega(G)$ is a set of vertices $X^* \supseteq X \cap V(G)$ of finite adhesion such that $\partial  X^* = \overline{X} \cap \Omega(G)$, where the closure $\overline{X}$ of $X$ is taken in $|G|$.

\begin{thm}
\label{thm:envelope}
Any set consisting of vertices and ends in a graph $G$ admits an envelope.
\end{thm}

\begin{proof}
Let $X \subseteq V(G) \cup \Omega(G)$ be a given set of vertices and ends in a graph $G$, and write $V(X):= X \cap V(G)$.
Let $\cR$ be an inclusionwise maximal set of pairwise disjoint rays of ends in $\overline{X}$. Put 
$$X':= V(X) \cup \bigcup_{R \in \cR} V(R)$$
and let $\mathcal{S}$ be the set of all centres of (infinite) stars attached to $X'$. We will show that 
$$X^* := X' \cup \mathcal{S}$$
is an envelope for $X$. 
The verification relies on the following two claims:

\begin{clm}
\label{claim_2}
If $S$ is a finite set of vertices and $C$ is a component of $G-S$ such that $X'\cap C$ is finite, then $X^* \cap C = X' \cap C$.
\end{clm}

Only $X^* \cap C \subseteq X' \cap C$ requires proof. For this consider some $v \in \mathcal{S}$. By definition, $v$ is the centre of an infinite star attached to $X'$. Since $S$ is finite and $X'$ meets $C$ finitely, it follows that $v \notin C$. Hence, $C \cap S = \emptyset$ and so $X^* \cap C = X' \cap C$ as claimed.

\begin{clm}
\label{claim_1}
If $S$ is a finite set of vertices and $C$ is a component of $G-S$ such that $\overline{C} \cap X = \emptyset$, then $ X^* \cap C$ is finite.
\end{clm}

To see the claim, consider some finite set of vertices $S$, and assume that $C$ is a component of $G-S$ such that $\overline{C}$ avoids $X$.
First, we show that $\overline{C} \cap \overline{X} = \emptyset$.
For this, observe that the set $\overline{C} \cup \mathring{E}_{1/2}(S,C)$ is open and disjoint from $X$ and so it is disjoint from $\overline{X}$.
In particular, $\overline{C}$ is disjoint from $\overline{X}$.
Hence every ray $R'\in\cR$ meets $C$ finitely. Furthermore, every ray from $\cR$ which meets $C$ also meets $S$, and since the rays in~$\cR$ are pairwise disjoint, at most $|S|$ rays from $\cR$ meet $C$. So $\bigcup_{R \in \cR} V(R)$ meets $C$ finitely, and hence so does $X'$.
By Claim~\ref{claim_2}, also $X^* \cap C = X' \cap C$ is finite. This establishes the claim.

To see $\partial  X^* = \overline{X} \cap \Omega(G)$, we show both inclusions separately. 
For $\supseteq$ consider any end $\eps \notin \partial X^*$. Then $C(S,\eps) \cap X^* = \emptyset$ for some finite set of vertices $S$. Consider a ray $R$ in $\eps$ that is completely contained in $C(S,\eps)$. Then $R$ is disjoint from any ray in $\cR$. By maximality of $\cR$, this means that $\eps \notin \overline{X}$.

For $\subseteq$ consider any end $\eps \notin \overline{X}$. Then there is a finite set of vertices $S$ such that $\hat{C}(S,\eps)$ avoids $X$. 
By Claim~\ref{claim_1}, also $X^*$ intersects $C(S,\eps)$ finitely, witnessing $\eps \notin \partial X^*$. 

To see that $X^*$ has finite adhesion, suppose for a contradiction that there is a component $C$ of $G-X^*$ with infinite neighbourhood.
Then by a routine application of the Star-Comb Lemma~\ref{lem_starcomb}, we either find a star or a comb attached to~$X^*$ whose centre $v$ or spine $R$ is contained in~$C$. The ray case results in an immediate contradiction as follows: If $\eps$ denotes the end with $R \in \eps$, then the comb attached to $X^*$ with spine $R$ witnesses that $\eps \in \partial X^*$. Since $\partial  X^* = \overline{X} \cap \Omega(G)$ by the earlier observation, we get $R\in \eps \in \overline{X}$. But then the existence of $R$ contradicts the maximality of $\cR$.

In the star case, note that for all finite sets of vertices $S$ disjoint from $v$, the component $C$ of $G-S$ containing $v$ meets $X^*$ infinitely. Then $C$ also meets $X'$ infinitely by Claim~\ref{claim_2}. But then it is straightforward to inductively construct a star with centre $v$ attached to $X'$, violating the maximality of $\mathcal{S}$. The completes the proof that $X^*$ is an envelope for $X$.
\end{proof}

Note that the envelopes constructed in  Theorem~\ref{thm:envelope} are in general neither connected nor end-faithful. But we can  easily obtain both properties with the following construction.

For a given subgraph $H \subseteq G$ of finite adhesion, we define a \emph{torso-extension} $H'\supseteq H$ as follows: First, we make $H$ induced. Then for each component $C$ of $G-H$, let $T_C \subseteq G[C \cup N(C)]$ be a finite tree such that all vertices from $N(C)$ are leaves of $T_C$. We add all these $T_C$ to $H$ to obtain $H'$.

\begin{lemma}
\label{lem_torso}
Let $G$ be connected. Whenever $H \subseteq G$ is a subgraph of finite adhesion, then every torso-extension $H'$ is an end-faithful connected subgraph of $G$ of finite adhesion with $\partial H' = \partial H$.
\end{lemma}

\begin{proof}
Since inside of each component of $G-H$ we only add a finite subgraph to $H$, also $H'$ has finite adhesion.

By construction, every vertex of $H' \setminus H$ is connected via a finite path in $H'$ to a vertex of $H$. Hence for connectivity of $H'$ it remains to show that there is a path in $H'$ between every two vertices $v,w \in H$.

Since $G$ is connected, there is a $v$--$w$ path $P$ in $G$. We consider $P$ as a sequence of edges between vertices of $H$ and segments inside of components $C$ of $G-H$ together with their end-vertices in $N(C)$. After replacing each of those segments in a component $C$ by a path in $T_C$ between the same end-vertices, we obtain a finite $v$--$w$ walk $P'$ contained in $H'$. So $H'$ is connected.

To see that $\partial H' = \partial H$, only $\subseteq$ requires proof. If $\omega \notin \partial H$, then $\omega$ lives in a unique component $C$ of $G-H$. Since $H' \cap C$ is finite it follows that $\omega$ also lives in a unique component $C'$ of $G-H'$ with $C' \subseteq C$ and hence $\omega \notin \partial H'$ by finite adhesion of $H'$.

We now argue that $H'$ contains an $\omega$-ray for every end $\omega$ in $\partial H' = \partial H$. Suppose without loss of generality that $H \neq \emptyset$ and fix any $\omega$-ray $R=r_0r_1r_2\ldots$ in $G$ with $r_0 \in V(H)$. By finite adhesion of $H$, the ray $R$ contains infinitely many vertices of $H$.
We will construct a ray $R' \subseteq H'$ that meets $R$ infinitely as follows: If $R \subseteq H'$, there is nothing to do. Otherwise, let $r_{n_0}$ be the first vertex on $R$ outside of $H'$, and consider the component $C_0 \ni r_{n_0}$ of $G-H$. Let $r_{k_0}$ be the last vertex of $R$ in $C_0$. Replace $r_{n_0-1}Rr_{k_0+1}$ by an $r_{n_0-1}$--$r_{k_0+1}$ path $P_0$ in $T_{C_0} \subseteq H'$ and call the  resulting ray $R_1$. Note that $R_1 \cap H \subseteq  R \cap H$. 
Now we iterate the same step for $R_1$ to find a new ray $R_2$ and so on. 
This yields a sequence of rays $R_1,R_2,R_3,\ldots$ with $R_n \cap H \subseteq  R \cap H$ and which agree on larger and larger initial segments contained in $H'$.
The union of these segments is a ray $R' \subseteq H'$ with $R' \cap H \subseteq R \cap H$, so $R'$ is an $\omega$-ray in $H'$ as desired.

To see that $H'$ is end-faithful, it remains to show that any two rays $R_1$ and $R_2$ in $H'$ that are equivalent in $G$ are also equivalent in $H'$. By assumption there is a collection $\cP$ of infinitely many disjoint $R_1$--$R_2$ paths in $G$. We will find infinitely many such paths in $H'$. Let $P$ be an $R_1$--$R_2$ path in $G$ with endvertices $r_1 \in R_1$ and $r_2 \in R_2$.  
As in the second paragraph, we find a $r_1$--$r_2$ walk $P'$ in $H'$. 

Consider the finitely many components of $G-H$ that meet $P'$ and delete from $\mathcal{P}$ all paths that meet one of these components -- by finite adhesion, $\mathcal{P}$ remains infinite. So we can find another $R_1$--$R_2$ path in $H'$ disjoint to the first one. Iterating this construction, we find infinitely many disjoint $R_1$--$R_2$ paths in $H'$, showing that $R_1$ and $R_2$ in are also equivalent in $H'$.
\end{proof}

A result for torsos of parts in tree-sets similar to Lemma~\ref{lem_torso} is proven in \cite[Section~2.6]{Torso}.

\begin{cor}\label{cor:envelope}
Any set consisting of vertices and ends in a connected graph $G$ has a connected, end-faithful envelope.
\end{cor}

Whenever we refer to \emph{the envelope} of $X$ inside a connected graph $G$, we assume that we fixed one possible end-faithful connected choice and call it $\cE_G(X)$.

\section{From topology to tree-decompositions}

In this section we employ the envelope technique in order to construct a tree-decomposition of finite adhesion adapted to some prescribed topological information. Roughly, given an infinite graph $G$ and an increasing sequence of closed subsets $X_0\subseteq X_1\subseteq X_2\subseteq\dots$ in $|G|$ such that $V(G) \subseteq \bigcup_{n \in \NN} X_n$, we construct a tree-decomposition $\mathcal{T}= (T,\mathcal{V})$ of finite adhesion such that precisely the ends of $X_n$ live in parts indexed by the first $n$ levels of $T$, and all other ends get displayed.\footnote{In the actual proof, we arrange for technical reasons that the ends of $X_n$ live precisely in parts indexed by the first $3n+3$ levels of $T$.}

However, we also want a device that ensures that all ends of some prescribed subcollection $\Delta$ of ends in $\bigcup_{n \in \NN} X_n$ live in pairwise distinct parts of $\mathcal{T}$. It turns out that this can be achieved provided that each $\Delta_n := \Delta \cap (X_{n} \setminus X_{n-1})$ is a discrete set.

\begin{lemma}\label{lem:finite_adhesion_sequence}
Let $G$ be a graph and $\Xi\subseteq\Omega(G)$. Suppose that there is a sequence $X_0\subseteq X_1\subseteq X_2\subseteq\dots$ of subsets of $V(G)\cup\Xi$ that are closed in $|G|$ with $V(G)\cup\Xi=\bigcup_{n\in\NN} X_n$. Denote $\Xi_n:=X_n\cap\Omega(G)$ and let $\Delta_n$ be a discrete subset of $\Xi_n\setminus\bigcup_{i<n}\Xi_i$ for all $n\in\NN$. Then there exists a sequence of induced subgraphs of finite adhesion $G_0\subseteq G_1\subseteq G_2\subseteq\dots$ of $G$ such that the following holds for all $n\in\NN$:
\begin{enumerate}[label=(\roman*)]
    \item\label{itemG:connected_parts} For every component $C$ of $G-G_n$, the set $(C\cap G_{n+1})\cup N(C)$ is connected in $G$;
    \item\label{itemG:ends2} $\partial G_{3n+2}\subseteq\Xi_n\setminus\Delta_n$;
    \item\label{itemG:delta_distinguished} for every component $C$ of $G-G_{3n+2}$, there is at most one end from $\Delta_n$ contained in $\partial C$;
    \item\label{itemG:contains_vertices} $X_n\cap V(G)\subseteq V(G_{3n+3})$;
    \item\label{itemG:ends3} $\partial G_{3n+3}=\Xi_n$.
\end{enumerate}
\end{lemma}

\begin{proof}
Set $G_0:=\emptyset$. We will inductively define subgraphs $G_0,G_1,\dots$ of $G$ all of finite adhesion so that $\ref{itemG:connected_parts}-\ref{itemG:ends3}$ are satisfied. 

Every step of the construction follows the same general pattern: To construct $G_{n+1}$ from $G_n$ consider the current set $\cC_n$ of components of $G-G_n$. For every $D\in\cC_n$ we consider the subgraph $\Tilde{D}:=G[D\cup N(D)]$ of $G$.
Each time we will define a set of vertices $V_D\subseteq V(\Tilde{D})$ of finite adhesion in $\Tilde{D}$ containing $N(D)$. Then also $G_{n+1}:=G_n\cup\bigcup_{D\in\cC_n}V_D$ has finite adhesion in $G$ since any component $C$ of $G-G_{n+1}$ is also a component of $\Tilde{D}-V_D$ for some $D\in\cC_n$.
Furthermore, we will make sure that $V_D$ is connected so that $\ref{itemG:connected_parts}$ is satisfied.

Next, we make two observations concerning the end space of $\Tilde{D}$, which both follow from the fact that $N(D)$ is finite:
Firstly, we have $\partial D=\partial\Tilde{D}$ in $G$, and secondly, the inclusion map $\iota$ as mentioned in Section \ref{section:end_spaces} is a homeomorphism from $\Omega(\Tilde{D})$ to $\partial D\subseteq|G|$.
Via this homeomorphism, we will in the following identify the spaces $\Omega(\Tilde{D})$ and $\partial D\subseteq|G|$.

Now for the actual construction of the sequence $G_0,G_1,\dots$, we proceed in steps of three. Suppose that $G_{3n}$ has already been defined. We demonstrate how to recursively construct
$$G_{3n} \rightsquigarrow G_{3n+1} \rightsquigarrow G_{3n+2} \rightsquigarrow G_{3n+3} = G_{3(n+1)}$$ in order to satisfy $\ref{itemG:connected_parts}-\ref{itemG:ends3}$ for the three indices $3n+1$, $3n+2$ and $3n+3$.

\textbf{1. Step $3n \rightsquigarrow 3n+1$.}\\
Let $D$ be any component from $\cC_{3n}$. Since $\Delta_n$ is discrete in $|G|$, also $\Delta_n\cap\partial{D}$ is discrete in $\partial{D}$.
Thus there is a set $\cO_D=\{O_\omega:\omega\in\Delta_n\cap\partial{D}\}$ of open subsets of $\partial{D}$ with $O_\omega\cap\Delta_n=\{\omega\}$ for all $\omega\in\Delta_n\cap\partial{D}$.
Applying Corollary~\ref{cor:envelope} we consider the envelope
\[V_D:=\cE_{\Tilde{D}}(((\Xi_n\cap\partial D)\setminus\bigcup\cO_D)\cup N(D))\]
which is a connected vertex set of finite adhesion in $\Tilde{D}$ (cf. Figure \ref{fig:1.step}).

\begin{figure}[h]
\centering
\begin{tikzpicture}

\tikzstyle{Xi_k-1}=[draw,star,color=green,fill=green,minimum size=6pt,inner sep=0pt]
\tikzstyle{Xi_k}=[draw,regular polygon, regular polygon sides=4,color=blue,fill=blue,minimum size=6pt,inner sep=0pt]
\tikzstyle{Delta_k}=[draw,regular polygon, regular polygon sides=3,color=red,fill=red,minimum size=6pt,inner sep=0pt]
\tikzstyle{other}=[draw,circle,color=black,fill=black,minimum size=4pt,inner sep=0pt]
\tikzstyle{T}=[draw,circle,color=black,fill=black,minimum size=2.5pt,inner sep=0pt]
\tikzstyle{O}=[draw,very thick,black!30]

\draw[thick] (-4,0) .. controls (-1.5,1.3) and (1.5,1.3) .. (4,0);
\node at (2.5,0.2){$G_{3n}$};

\draw[thick] (-4,5) .. controls (-3,-0.2) and (3,-0.2) .. (4,5);
\node[rotate=60] at (3,3.2){$D\in\cC_{3n}$};

\node[Xi_k] at (-0.7,1.7) {};
\node[Xi_k] at (0,1.4) {};
\node[Xi_k] at (0.7,1.7) {};

\node[Xi_k-1] at (-2.3,0.1) {};
\node[Xi_k-1] at (-1.5,0.4) {};
\node[Xi_k-1] at (-0.6,0) {};
\node[Xi_k-1] at (0.7,0.1) {};
\node[Xi_k-1] at (1.6,0.3) {};

\node[Xi_k] at (-1.3,3.2) {};
\node[Delta_k] at (-1.1,3.8) {};
\node[Xi_k] at (-0.1,3.3) {};
\node[Delta_k] at (0.3,3.8) {};
\node[Xi_k] at (0.7,3.3) {};
\node[other] at (0,4.6) {};
\node[other] at (0.8,4.4) {};
\node[other] at (1.7,3.8) {};
\node[other] at (2.5,3.8) {};
\node[other] at (1.9,4.4) {};

\node[Delta_k] at (-2.2,2.7) {};
\node[Xi_k] at (-2.8,3.1) {};
\node[other] at (-2.3,3.5) {};
\node[other] at (-2.9,3.7) {};

\draw[O] (0.3,3.9) ellipse (23pt and 35pt);
\draw[O,rotate around={75:(-0.8,3.4)}] (-0.8,3.4) ellipse (20pt and 27pt);
\draw[O,rotate around={35:(-2.5,3.2)}] (-2.5,3.2) ellipse (16pt and 27pt);
\node[black!40] at (-1.9,4.2){$\cO_D$};

\node[draw,align=left] at (0,-2) {\,\,\,\,\,\,\,ends from $\Xi_{n-1}$\\ \,\,\,\,\,\,\,ends from $\Delta_n$\\ \,\,\,\,\,\,\,ends from $\Xi_n\setminus (\Xi_{n-1}\cup\Delta_n)$\\ \,\,\,\,\,\,\,ends from $\Omega(G)\setminus\Xi_n$};
\node[Xi_k-1] at (-2.38,-1.17) {};
\node[Delta_k] at (-2.38,-1.73) {};
\node[Xi_k] at (-2.38,-2.27) {};
\node[other] at (-2.38,-2.81) {};

\draw[dashed] (0,1.5) ellipse (30pt and 24pt);
\node at (0,1.9){$V_D$};

\end{tikzpicture}
\caption{Construction step $3n \rightsquigarrow 3n+1$.}
\label{fig:1.step}
\end{figure}
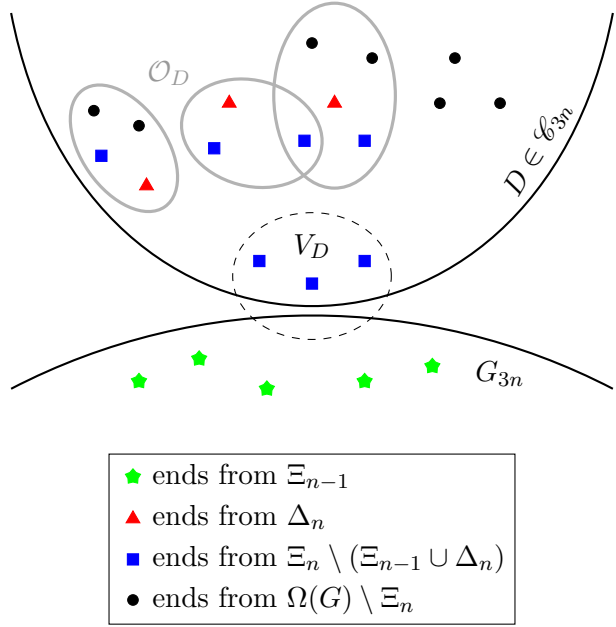

We now determine which ends are contained in $\partial V_D$.
Since both $X_n$ and $\Omega(G)$ are closed in $|G|$, also $\Xi_n=X_n\cap\Omega(G)$ is closed in $|G|$.
Hence $(\Xi_n\cap\partial D)\setminus\bigcup\cO_D$ is closed in the subspace $\partial D$ of $|G|$.
Since $N(D)$ is finite and therefore does not have any ends in its closure, it follows from the definition of an envelope and our identification of $\Omega(\Tilde{D})$ with $\partial D$ that
\[\partial V_D=\overline{(((\Xi_n\cap\partial D)\setminus\bigcup\cO_D)\cup N(D))}\cap\partial D=\overline{(\Xi_n\cap\partial D)\setminus\bigcup\cO_D}=(\Xi_n\cap\partial D)\setminus\bigcup\cO_D.\]

Hence by \ref{itemG:ends3} for $G_{3n}$, the graph $G_{3n+1}=G_{3n}\cup\bigcup_{D\in\cC_{3n}}V_D$ satisfies
\begin{enumerate}[label=(\roman*)]\setcounter{enumi}{5}
\item\label{itemG:ends1} $\partial G_{3n+1}\subseteq\Xi_n\setminus\Delta_n$.
\end{enumerate}
Next, we show that
\begin{enumerate}[label=(\roman*),resume]
\item\label{itemG:open_cover} for every component $C$ of $G-G_{3n+1}$, there is an open cover $\cO$ of $\partial C$ such that each set from $\cO$ contains at most one end from $\Delta_n$.
\end{enumerate}
Let $C$ be a component of $G-G_{3n+1}$ and $D'$ the component of $G-G_{3n}$ with $C\subseteq D'$. We show that $\ref{itemG:open_cover}$ is fulfilled with
\[\cO:=\{O\cap\partial C:O\in\cO_{D'}\}\cup\{\partial C\setminus\Xi_n\}.\]
Clearly, all sets in $\cO$ are open in $\partial C$ and contain at most one end from $\Delta_n$.
For the proof that $\partial C\subseteq\bigcup\cO$, we observe that $C$ and $V_{D'}$ are disjoint and the neighbourhood of $C$ is finite.
Therefore, $\partial C$ and $\partial V_{D'}$ are disjoint.
Since $\partial V_{D'}=(\Xi_n\cap\partial D')\setminus\bigcup\cO_{D'}$, we have $\partial C\cap\Xi_n\subseteq\bigcup\cO_{D'}$ and therefore $\partial C\subseteq\bigcup\cO$.

\textbf{2. Step $3n+1 \rightsquigarrow 3n+2$.}\\
Let $D$ be any component from $\cC_{3n+1}$. By $\ref{itemG:open_cover}$ there exists an open cover $\cO$ of $\partial D$ such that each set from $\cO$ contains at most one end from $\Delta_n$ (cf. Figure \ref{fig:2.step}).
Then by Theorem \ref{thm:RNT}, there is a rayless normal tree $T$ in $\Tilde{D}$ such that for every component $C'$ of $\Tilde{D}-T$ there is a set $O\in\cO$ with $\partial C'\subseteq O$.
By Lemma \ref{lem:extend_RNT} there exists a rayless normal tree $T^*$ in $\Tilde{D}$ such that $V(T)\cup N(D)\subseteq V(T^*)$.
We define $V_D:=V(T^*)$.
Then every component $C$ of $G-V_D$ is contained in a component $C'$ of $G-T$ and thus there is a set $O\in\cO$ with $\partial C\subseteq O$.
Then by $\ref{itemG:open_cover}$, $C$ contains at most one end from $\Delta_n$.
Hence $G_{3n+2}=G_{3n+1}\cup\bigcup_{D\in\cC_{3n+1}}V_D$ satisfies $\ref{itemG:delta_distinguished}$.
Furthermore, $T^*$ has finite adhesion in $\Tilde{D}$ by Lemma~\ref{lem:extend_RNT}.
Finally, normal trees are end-faithful by \cite[Lemma~8.2.3]{diestel2015book}, so from the fact that $T^*$ is rayless it follows that $\partial T^*=\emptyset$.
Therefore $\partial G_{3n+2}=\partial G_{3n+1}$ and $\ref{itemG:ends2}$ is a consequence of $\ref{itemG:ends1}$.

\begin{figure}[h]
\centering
\begin{tikzpicture}

\tikzstyle{Xi_k-1}=[draw,star,color=green,fill=green,minimum size=6pt,inner sep=0pt]
\tikzstyle{Xi_k}=[draw,regular polygon, regular polygon sides=4,color=blue,fill=blue,minimum size=6pt,inner sep=0pt]
\tikzstyle{Delta_k}=[draw,regular polygon, regular polygon sides=3,color=red,fill=red,minimum size=6pt,inner sep=0pt]
\tikzstyle{other}=[draw,circle,color=black,fill=black,minimum size=4pt,inner sep=0pt]
\tikzstyle{T}=[draw,circle,color=black,fill=black,minimum size=2.5pt,inner sep=0pt]
\tikzstyle{O}=[draw,very thick,black!30]

\draw[thick] (-4,0) .. controls (-1.5,1.3) and (1.5,1.3) .. (4,0);
\node at (2.5,0.2){$G_{3n+1}$};

\draw[thick] (-4,5) .. controls (-3,-0.2) and (3,-0.2) .. (4,5);
\draw (-1,5) .. controls (-1,1.5) and (1,1.5) .. (1,5);
\draw (1.3,5) .. controls (0.5,1.3) and (2.2,1.3) .. (3.3,5);
\draw (-1.3,5) .. controls (-0.5,1.3) and (-2.2,1.3) .. (-3.3,5);
\node[rotate=60] at (3,3.2){$D\in\cC_{3n+1}$};
\draw[thick,rotate=25] (-4,3.7) .. controls (-4,1.3) and (-2,1.3) .. (-2,3.7);

\node[Xi_k] at (-0.7,0.5) {};
\node[Xi_k] at (0,0.2) {};
\node[Xi_k] at (0.7,0.5) {};

\node[Xi_k-1] at (-2.3,0) {};
\node[Xi_k-1] at (-1.5,0.3) {};
\node[Xi_k-1] at (-0.6,-0.1) {};
\node[Xi_k-1] at (0.7,0) {};
\node[Xi_k-1] at (1.6,0.2) {};

\node[Xi_k] at (-1.6,3.2) {};
\node[Delta_k] at (-1.4,3.8) {};
\node[Xi_k] at (-0.4,3.3) {};
\node[Delta_k] at (0,3.8) {};
\node[Xi_k] at (0.4,3.3) {};
\node[other] at (-0.3,4.6) {};
\node[other] at (0.5,4.4) {};
\node[other] at (1.5,3.8) {};
\node[other] at (2.3,3.8) {};https://de.overleaf.com/project/61671a741bd43f87bb979c21
\node[other] at (1.7,4.4) {};

\node[Delta_k] at (-3.6,0.8) {};
\node[Xi_k] at (-4.2,1.2) {};
\node[other] at (-3.7,1.6) {};
\node[other] at (-4.3,1.8) {};

\node at (0,1.8){$T^*=V_D$};
\node[T] (a) at (0.1,1.3) {};
\node[T] (b) at (-0.7,1.4) {};
\node[T] (c) at (-1.3,1.6) {};
\node[T] (d) at (-1.1,2) {};
\node[T] (e) at (0.9,1.5) {};
\node[T] (f) at (1.4,1.7) {};
\node[T] (g) at (1,2) {};
\node[T] (h) at (-0.5,0.8) {};
\node[T] (i) at (0,0.8) {};
\node[T] (j) at (0.4,0.7) {};
\draw (a) -- (b);
\draw (b) -- (c);
\draw (b) -- (d);
\draw (a) -- (e);
\draw (e) -- (f);
\draw (e) -- (g);
\draw (b) -- (h);
\draw (a) -- (i);
\draw (a) -- (j);

\draw[O] (0,3.9) ellipse (21pt and 35pt);
\draw[O,rotate around={75:(-1.1,3.4)}] (-1.1,3.4) ellipse (20pt and 27pt);
\draw[O,rotate around={75:(1.1,4.2)}] (1.1,4.2) ellipse (16pt and 48pt);
\node[black!40] at (-2.1,4.1){$\cO$};

\node[draw,align=left] at (0,-2) {\,\,\,\,\,\,\,ends from $\Xi_{n-1}$\\ \,\,\,\,\,\,\,ends from $\Delta_n$\\ \,\,\,\,\,\,\,ends from $\Xi_n\setminus (\Xi_{n-1}\cup\Delta_n)$\\ \,\,\,\,\,\,\,ends from $\Omega(G)\setminus\Xi_n$};
\node[Xi_k-1] at (-2.38,-1.17) {};
\node[Delta_k] at (-2.38,-1.73) {};
\node[Xi_k] at (-2.38,-2.27) {};
\node[other] at (-2.38,-2.81) {};

\end{tikzpicture}
\caption{Construction step $3n+1 \rightsquigarrow 3n+2$.}
\label{fig:2.step}
\end{figure}

\textbf{3. Step $3n+2 \rightsquigarrow 3n+3$.}\\
Again let $D$ be any component from $\cC_{3n+2}$ (the components of $G-G_{3n+2}$). We define
\[V_D:=\cE_{\Tilde{D}}((X_n\cap \overline{D})\cup N(D)).\]
Then it follows from the definition of an envelope that
\[X_n\cap V(\Tilde{D})\subseteq ((X_n\cap \overline{D})\cup N(D))\cap V(\Tilde{D})\subseteq V_D.\]
Therefore $G_{3n+3}=G_{3n+2}\cup\bigcup_{D\in\cC_{3n+2}}V_D$ satisfies $\ref{itemG:contains_vertices}$.
Furthermore, since $N(D)$ is finite and $X_n$ is closed, we have
\[\partial V_D=\overline{((X_n\cap \overline{D})\cup N(D))}\cap\partial D=X_n\cap\partial D=\Xi_n\cap\partial D.\]
Then together with $\ref{itemG:ends2}$ we obtain $\partial G_{3n+3}=\Xi_n$ which proves $\ref{itemG:ends3}$.
\end{proof}

\begin{thm}
\label{thm:topology_to_td}
Let $G$ be a connected graph and $\Xi\subseteq\Omega(G)$. Suppose that there is a sequence $X_0\subseteq X_1\subseteq X_2\subseteq\dots$ of subsets of $V(G)\cup\Xi$ that are closed in $|G|$ with $V(G)\cup\Xi=\bigcup_{n\in\NN} X_n$. Denote $\Xi_n:=X_n\cap\Omega(G)$ and let $\Delta_n$ be a discrete subset of $\Xi_n\setminus\bigcup_{i<n}\Xi_i$ for all $n\in\NN$.
Then there is an upwards connected tree-decomposition $\script{T}=(T,\cV)$ of finite adhesion with connected parts which homeomorphically displays $\Omega(G)\setminus\Xi$ such that the boundary of every part contains at most one end from $\bigcup_{n\in\NN}\Delta_n$.
\end{thm}

\begin{proof}
Let $G_0 \subseteq G_1 \subseteq G_2 \subseteq \dots$ be the sequence from Lemma~\ref{lem:finite_adhesion_sequence} with properties $\ref{itemG:connected_parts}-\ref{itemG:ends3}$ and suppose without loss of generality that $G_0 = \emptyset$. This sequence gives rise to a tree-decomposition $\cT=\p{T,\cV}$ of finite adhesion and into connected parts as follows: Write $\cC_n$ for the set of components of $G-G_n$.
We define a tree order $\leq_T$ on $T:=\bigsqcup_{n\in\NN}\cC_n$ as follows: For all $C_n\in\cC_n$ and $C_m\in\cC_m$, let $C_n\leq_T C_m$ if and only if $C_n\supseteq C_m$ and $n\leq m$; this will be our decomposition tree.
Note that $G_0 = \emptyset$ ensures $T$ has a root whose associated part is $G$.
The part corresponding to a node $C \in \cC_n$ of $T$ will be $N(C) \cup \p{C \cap G_{n+1}}$ (which is precisely the set $V_C$ from the proof of Lemma \ref{lem:finite_adhesion_sequence}).
Then it is readily checked that all properties (T1) -- (T3) of a tree-decomposition are satisfied, in particular (T1) holds by $\ref{itemG:contains_vertices}$. All parts of $\cT$ are connected by $\ref{itemG:connected_parts}$.

It is clear from the construction that $\cT$ is upwards connected.
Furthermore, by $\ref{itemG:ends3}$ the interior of $\cT$ is $\Xi$ and hence its boundary is $\Omega(G)\setminus\Xi$.
Therefore $\cT$ homeomorphically displays $\Omega(G)\setminus\Xi$ by Lemma~\ref{lem_upwardsconndisplay}.

It is left to show that in every part of $\cT$ there lives at most one end from $\bigcup_{n\in\NN}\Delta_n$.
For any $n\in\NN$, we have $\Delta_n\subseteq\partial G_{3n+3}\setminus\partial G_{3n+2}$ by $\ref{itemG:ends2}$ and $\ref{itemG:ends3}$.

Since this inclusion holds for all $n\in\NN$, it follows that $\partial G_{3n+3}\setminus\partial G_{3n+2}$ does not contain ends from $\Delta_{n'}$ for any $n'\neq n$.
Furthermore, by $\ref{itemG:delta_distinguished}$ every component in $\cC_{3n+2}$ contains at most one end from $\Delta_n$ in its boundary.
Hence all ends from $\Delta_n$ are contained in the boundaries of parts of the form $N(C)\cup(C\cap G_{3n+3})$ for $C\in\cC_{3n+2}$, and in the boundary of every such part there is no end from $\Delta_{n'}$ for any $n'\neq n$ and at most one end from $\Delta_n$.
This finishes the proof.
\end{proof}

\section{Tree-decompositions displaying sets of ends}
In this section we will prove our characterisation announced in Theorem~\ref{thm_main3} of \emph{displayable} subsets of $\Omega(G)$, i.e.\ subsets which can be (homeomorphically) displayed by a tree-decomposition of finite adhesion.

\begin{thm}\label{thm:display_sets}
For any connected graph $G$ and any set $\Psi$ of ends of $G$ the following are equivalent:
\begin{enumerate}
    \item\label{item_thm:display_sets_1} There is an upwards connected tree-decomposition of finite adhesion with connected parts that homeomorphically displays $\Psi$.
    \item There is a tree-decomposition of finite adhesion displaying $\Psi$.
    \item There is a tree-decomposition of finite adhesion with boundary $\Psi$.
    \item $|G|_\Psi$ is completely metrizable.
    \item\label{item_thm:display_sets_5} $\Psi$ is $G_\delta$ in $|G|$.
\end{enumerate}
\end{thm}

\begin{proof}
We demonstrate the following sequence of implications:

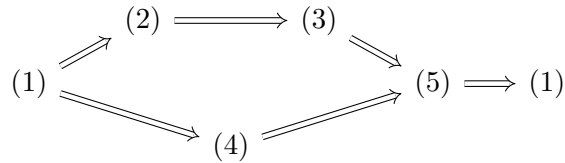
\begin{center}
\begin{tikzcd}[row sep=tiny, column sep=0.9em]
&& (2) \arrow[rr, Rightarrow] && (3) \arrow[drr, Rightarrow] && && \\
(1) \arrow[urr, Rightarrow] \arrow[drrr, Rightarrow] && && && (5) \arrow[rr, Rightarrow] && (1) \\
&& & (4) \arrow[urrr, Rightarrow] & && &&
\end{tikzcd}
\end{center}

The implications $(1)\Rightarrow (2) \Rightarrow (3)$ are trivial.

$(1) \Rightarrow (4)$:
Let $(T,\cV)$ be a tree-decomposition of finite adhesion of $G$ homeomorphically displaying $\Psi$ with a fixed root $r$ of $T$. We begin by defining a complete metric $d_T$ on $V(T) \cup \Omega(T)$. Assign to every $e\in E(T)$ a number $\ell(e)$:
If $e\in E(T)$ is a $T^n$--$T^{n+1}$ edge (i.e. an edge between level $n$ and level $n+1$ of $T$), we set $\ell(e)=1/2^n$.
If $P$ is a (possibly infinite) path in $T$, we say that the finite number $\sum_{e\in E(P)}\ell(e)$ is the \emph{length} of $P$.
Now we define $d_T(x,y)$ for all $x,y\in V(T) \cup \Omega(T)$:
If $x$ and $y$ are both vertices, let $d_T(x,y)$ be the length of the unique $x$--$y$ path in $T$.
If $x$ is a vertex and $y$ is an end, then let $d_T(x,y)$ be the length of the unique ray from $y$ which starts in $x$.
Similarly, if both $x$ and $y$ are ends, let $d_T(x,y)$ be the length of the unique double ray in $T$ between $x$ and $y$.
It is straight-forward to check that $d_T$ defines a complete metric on $V(T) \cup \Omega(T)$.

We now use $d_T$ to define a metric $d$ on $G \cup \Psi$.
For every vertex $v \in V(G)$, let $v_T$ be the least vertex of $T$ with respect to the tree order such that $v$ is contained in the part $V_{v_T}$ (this is well-defined according to (T3)).
Additionally, for every end $\omega\in\Psi$, let $\omega_T$ be the end of $T$ which $\omega$ corresponds to.
For all $x,y\in V\cup\Psi$, we define
$$
d(x,y) =
\begin{cases*}
    0 & if $x=y$,\\
    1/2^n & if $x\neq y\in V(G)$ and $x_T=y_T$ lies in the $n$th level of $T$,\\
    d_T(x_T,y_T) & if $x_T\neq y_T.$
\end{cases*}
$$

Next, we prove that $d$ is a metric on $V(G)\cup\Psi$.
It is clear that $d(x,x) = 0$ and $d(x, y) > 0$ for all $x \neq y$ and that $d$ is symmetric.
We show that triangle inequality holds: Let $x,y,z$ be pairwise distinct elements of $V(G)\cup\Psi$.
We need to show that
\begin{equation}
d(x,z)\leq d(x,y)+d(y,z). \tag{$*$}\label{eq:triangle-inequality}
\end{equation}

Clearly, (\ref{eq:triangle-inequality}) holds if $x_T=y_T=z_T$.
If $x_T=y_T\neq z_T$, then $d(x,z)=d(y,z)$ and hence (\ref{eq:triangle-inequality}) follows. A similar argument works if $y_T=z_T$.
Next, suppose that $x_T=z_T\neq y_T$ and let $n$ be the level of $x_T$ in $T$. Then $d(x,z)=1/2^n$ and since $\ell(e)\geq 1/2^n$ for every edge $e$ of $T$ with endvertex $x_T$ also $d(x,y)\geq 1/2^n$, which proves (\ref{eq:triangle-inequality}). Finally, if $x_T$, $y_T$ and $z_T$ are pairwise distinct, then (\ref{eq:triangle-inequality}) follows from the triangle inequality for $d_T$. This finishes the proof of (\ref{eq:triangle-inequality}).

For the proof that $d$ is complete, let $(x_n)_{n\in\NN}$ be a Cauchy-sequence in $V(G)\cup\Psi$. Hence $((x_n)_T)_{n\in\NN}$ is a Cauchy-sequence in $T$ because $d_T(v_T,w_T)\leq d(v,w)$ for all $v,w\in V(G)\cup\Psi$.
If $((x_n)_T)_{n\in\NN}$ is eventually constant, then $(x_n)_{n \in \NN}$ is eventually contained in $V_t$ for some $t\in V(T)$.
If $t$ lies in the $n$th level of $T$, then $d(v,w)\geq 1/2^n$ for all $v\neq w\in V_t$. Hence also $(x_n)_{n\in\NN}$ is eventually constant.
Otherwise, if $((x_n)_T)_{n\in\NN}$ is not eventually constant, then $((x_n)_T)_{n\in\NN}$ converges to an end $\omega$ of $T$ and thus $(x_n)_{n\in\NN}$ converges to the end of $G$ which corresponds to $\omega$.
Finally, we extend $d$ to a complete metric on $G\cup\Psi$ by relating every edge $vw$ of $G$ linearly to a real closed interval of length $d(v,w)$. We omit the details.

It is left to show that the metric $d$ induces the subspace topology on $G\cup\Psi$ inherited from $|G|$.
We need to show for any given $x\in G\cup\Psi$ that
\begin{enumerate}[label=($\dagger$)]
\item every \textsc{MTop}-basic open neighbourhood of $x$ in $G\cup\Psi$ contains an open $\eps$-ball around $x$ with respect to $d$, and vice versa. \label{item:metric-induces-topology}
\end{enumerate}
This is clear if $x$ is an inner point of an edge. Next, let $x\in V(G)$ be a vertex and $n$ the level of $x_T$ in $T$.
Then \ref{item:metric-induces-topology} is true because every edge of $G$ which has $x$ as an endvertex has length at least $1/2^n$ and at most $1$.

Now suppose that $x\in\Psi$ and let $\hat{C}_\eps(S,x)$ be a basic open neighbourhood of $x$ in $|G|$ for some $\eps \leq 1$.
Let $n$ be the maximum level of $T$ containing a vertex $s_T$ for some $s\in S$.
We show that the open ball $B$ in $|G|$ with respect to $d$ with radius $\eps/2^n$ and centre $x$ is a subset of $\hat{C}_\eps(S,x)$.
First, consider the open ball $B'$ in $T$ with respect to the metric $d_T$ with radius $\eps/2^n$ and centre $x'$, where $x'$ is the end of $T$ which $x$ corresponds to.
Let $e$ be the edge of $T$ which is contained in the normal $x'$-ray in $T$ and connects a node $u_n$ form the $n$th level of $T$ to a node $u_{n+1}$ from the $n+1$st level.
Then $B'$ is completely contained in the closure of the component $D$ of $T-e$ with $u_{n+1}\in V(D)$ since
\[d_T(u_{n+1},x')=\sum_{i\geq n+1}1/2^i=1/2^n\geq\eps/2^n.\]
In particular, every vertex in $B'$ lies in the $n+1$st level of $T$ or above.
Next, it follows from the definition of the metric $d$ that every vertex in $B$ is contained in a part $V_t$ with $t\in B'\subseteq \overline{D}$, but no vertex of $B$ can be contained in a part $V_t$ such that the level of $t$ in $T$ is at most $n$.
Therefore all vertices in $B$ and similarly also all ends in $B$ are contained in $\overline{H_e}$, where $H_e$ is the subgraph of $G$ from the definition of upwards connectedness.
Since $H_e$ is disjoint from $S$, connected by upwards connectedness of $\cT$, and $x$ orients $e$ towards $x'$, we have $\overline{H_e}\subseteq \hat{C}(S,\omega)$.
Hence all vertices and ends in $B$ and all edges with both endvertices in $B$ are contained in $\hat{C}_\eps(S,x)$; it is left to show the same for points of edges in $B$ with only one endvertex in $B$.
Every such edge $f$, however, has its other endvertex in $V_{u_n}$ by (T3), and as $u_n$ lies in the $n$th level of $T$, the length of $f$ with respect to $d$ is at least $1/2^n$.
Recall that any point $p$ on $f$ in $B$ has distance less than $\eps/2^n$ to $x$ and therefore also to the end vertex of $f$ in $B$.
Thus $p$ is contained in $\hat{C}_\eps(S,x)$, as desired.

Conversely, let $B$ be an open $\eps$-ball around $x$ with respect to $d$ of radius $0<\eps\leq 1$.
Let $\omega\in\Omega(T)$ be the end of $T$ corresponding to $x$ and $R$ the rooted $\omega$-ray in $T$.
Choose $n\in\NN$ such that $1/{2^n}<\eps$ and let $t^i\in V(T)$ be the node in $R\cap T^i$ for $i\in\{n+2,n+3\}$.
Then define $S$ as the separator induced by the edge $t^{n+2}t^{n+3}$ of $T$ in $G$.
Now $C:=\hat{C}_{1/2^{n+1}}(S,x)$ is a subset of $B$:
Let $y$ be any point in $C$; we have to show that $d(y,x)<\eps$.
First suppose that $y\in C(S,x)$ and let $w$ be a vertex from the part $V_{t^{n+3}}$.
For any point $z\in C(S,x)$ we have
$$d(w,z)\leq\sum_{i\geq n+3}1/2^i=1/2^{n+2}.$$
Hence
$$d(y,x)\leq d(y,w)+d(w,x)\leq 1/{2ˆ{n+2}} + 1/{2ˆ{n+2}} = 1/2^{n+1}<\eps.$$
Next, suppose that $y$ is an inner point of an $S$--$C(S,x)$ edge with endvertex $v$ in $C(S,x)$.
We have seen above that $d(v,x)\leq 1/2^{n+1}$.
Hence it follows from the choice of $C$ that
$$d(y,x)\leq d(y,v)+d(v,x)\leq 1/2^{n+1}+1/2^{n+1}<\eps$$
which proves $C\subseteq B$.

$(4)\Rightarrow (5)$:
Assume that $|G|_\Psi$ is completely metrizable.
We claim that
\begin{itemize}
    \item $\Psi$ is $G_\delta$ in $|G|_\Psi$, and
    \item $|G|_\Psi$ is $G_\delta$ in $|G|$.
\end{itemize}
This implies (5) as being $G_\delta$ is transitive.

Since closed subsets of metrizable spaces are always $G_\delta$ \cite[Corollary~4.1.12]{engelking1989book}, we get that $\Psi$ is $G_\delta$ in $|G|_\Psi$.
Next, by a well-known result of \v{C}ech \cite[Theorem~4.3.26]{engelking1989book} all completely metrizable spaces, and so in particular  $|G|_\Psi$, are \v{C}ech-complete, and by \cite[Exercise 3.9.A]{engelking1989book}, all \v{C}ech-complete spaces are $G_\delta$ in their closures. Thus we conclude that $|G|_\Psi$ is $G_\delta$ in its closure $|G|$.

$(3) \Rightarrow (5)$:
Let $(T,\cV)$ be a tree-decomposition of finite adhesion of $G$ with boundary $\Psi$. Fix a root $r$ of $T$ and denote by $E_n$ the set of all edges between the $n$th and $n+1$st level of $T$. For every edge $e\in E_n$, let $(A_e,B_e)$ be the respective separation of $G$ such that $V_r\subseteq A_e$ and let $S_e=A_e\cap B_e$ be the corresponding finite adhesion set. Note that $A_e$ contains every part $V_t$ with $t \in T^{\leq n}$.
We denote
$$\cC_e:= \bigcup\{\hat{C}_{1/2}(S_e,\omega):\omega\in\partial B_e\}.$$
Then $O_n:=\bigcup_{e\in E_n}\cC_e$ is an open set in $|G|$ because it is a union of open sets. We show that $\Psi=\bigcap_{n\in\NN}O_n$. Clearly, $\Psi\subseteq\bigcap_{n\in\NN}O_n$. For the converse inclusion, let $\omega\in\bigcap_{n\in\NN}O_n$.
We show that $\omega$ does not live in any part of $(T,\cV)$ and therefore lies in the boundary of $(T,\cV)$.
Indeed, if $\omega\in\partial V_t$ for $t \in T^n$, 
then $\omega$ is not contained in $O_{n+1}$, a contradiction.

$(5)\Rightarrow(1)$:
Let $\Psi\subseteq\Omega(G)$ be a $G_\delta$ set in $|G|$.
Hence $G\cup\Xi$ where $\Xi:=\Omega(G)\setminus\Psi$ is an $F_\sigma$ set in $|G|$ and by Lemma \ref{lem_Fsigma}, also $V(G)\cup\Xi$ is an $F_\sigma$ set in $|G|$.
This means that $V(G)\cup\Xi=\bigcup_{n\in\NN}X_n$ is a countable union of sets $X_n$ which are closed in $|G|$, we may assume that $X_0\subseteq X_1\subseteq\cdots$.
By applying Theorem~\ref{thm:topology_to_td} (with $\Delta_n = \emptyset$) there is an upwards connected tree-decomposition of finite adhesion into connected parts that homeomorphically displays $\Psi = \Omega(G) \setminus \bigcup_{n\in\NN} X_n$.
\end{proof}

\begin{cor}\label{end-space-of-tree}
Displayable sets of ends are completely metrizable.
\end{cor}

\begin{proof}
The implication $(2) \Rightarrow (4)$ in Theorem~\ref{thm:display_sets} says that for every displayable set of ends $\Psi \subseteq \Omega(G)$ in a graph $G$ we have that $|G|_\Psi$ is completely metrizable. Since $\Psi \subseteq |G|_\Psi$ is closed, and closed subspaces of completely metrizable spaces are again completely metrizable, it follows that $\Psi $ is completely metrizable.
\end{proof}

\begin{cor}
Let $G$ be a graph with a displayable set of ends $\Psi \subseteq \Omega(G)$ and let $\Phi$ be a subset of $\Psi$. Then $\Phi$ is (homeomorphically) displayable if and only if $\Phi$ is a $G_\delta$ set in $\Psi$.
\end{cor}

\begin{proof}
Immediate from (2) $\Leftrightarrow$ (5) in Theorem~\ref{thm:display_sets} and transitivity of the $G_\delta$-property.
\end{proof}

\begin{cor}
Let $G$ be a graph with a normal spanning tree. Then a subset $\Phi \subseteq \Omega(G)$ is (homeomorphically) displayable if and only if $\Phi$ is a $G_\delta$ set in $\Omega(G)$.
\end{cor}

\begin{proof}
Follows from (\ref{item_Psi=Omegaiv}) $\Rightarrow$ (\ref{item_Psi=Omegaiii}) in Theorem~\ref{thm_Psi=Omega} together with the previous corollary for $\Psi=\Omega(G)$.
\end{proof}

\section{Tree-decompositions distributing sets of ends}
\label{sec_8}

In this section we characterise which subsets of ends can be distributed by a tree-decomposition of finite adhesion.
Recall that  a topological space $X \subseteq Z$ has a \emph{$\sigma$-discrete expansion in $Z$} if it can be written as a disjoint union $X = \bigsqcup_{n\in\NN} X_n$ such that all $X_n$ are discrete and all $Y_n:= \bigcup_{i \leq n} X_i$ are closed in $Z$.

\begin{thm}\label{itemDistr}
Let $G$ be a connected graph and $\Xi \subseteq \Omega(G)$ a subset of ends of $G$. Then the following are equivalent:
\begin{enumerate}[label=(\roman*)]
    \item There is a tree-decomposition of finite adhesion distributing $\Xi$.
    \item $V(G)$ is a countable union of slender vertex sets $U_n$ such that $\bigcup_{n \in \mathbb{N}} \partial U_n = \Xi$.
    \item $V(G) \cup \Xi $ has a $\sigma$-discrete expansion in $|G|$.
    \item There is an upwards connected tree-decomposition of finite adhesion with connected parts realising $(\Xi,\Xi^\complement)$.
    \end{enumerate}
\end{thm}

\begin{proof}
We will show a cyclic chain of implications. For $(i) \Rightarrow (ii)$, 
suppose we have a tree-decomposition $(T,\mathcal{V})$ with root $r$ of finite adhesion that distributes $\Xi$.

We define
$$U_n = \bigcup_{t \in T^{\leq n}} V_{t}.$$
By property (T1) of a tree-decomposition, it is clear that 
$V(G) \subseteq \bigcup_{n \in \mathbb{N}} U_n$.
Since $\Xi$ is the interior of $(T,\cV)$, we also have $\Xi=\bigcup_{n \in \mathbb{N}} \partial U_n$ as desired.

Furthermore, each $U_n$ is slender:
Clearly, all vertices are isolated in $|G|$.
Additionally, $\partial U_n \setminus \partial U_{n-1}$ consists of at most one end for each part $V_t$ for $t \in T^n$ and hence all ends in $\partial U_n \setminus \partial U_{n-1}$ are isolated points of $U_n$.
Therefore, each $\overline{U_n}$ has Cantor-Bendixson rank at most $n+1$ by induction.

For $(ii) \Rightarrow (iii)$, suppose $V(G)$ is a countable union of slender vertex sets $U_n$ such that $\bigcup_{n \in \mathbb{N}} \partial U_n = \Xi$. Without loss of generality, the sequence of the $U_n$ is increasing. Write $X_n = \overline{U_n}$ and let $Y_0=X_0$ and $Y_{n+1} = X_{n+1} \setminus X_n$.
By assumption, each $Y_n$ has finite Cantor-Bendixson rank say $k_n$.
Recall that $Y_n^{(0)} := Y_n$ and $Y_n^{(i+1)}$ denotes the derived space of $Y_n^{(i)}$ for all $i \in \NN$.
Since $Y_n$ has rank $k_n$, we have $Y_n^{(k_n)} = \emptyset$.
Let $Z_{n,i}:= Y_n^{(i)} \setminus Y_0^{(i+1)}$ be the subset of $Y_n$ consisting of all elements that get deleted when forming $Y_n^{(i+1)}$ for  $0 \leq i \leq k_n-1$. We claim that
$$Z_{0,k_0-1},Z_{0,k_0-2},\ldots,Z_{0,0},Z_{1,k_1-1},Z_{1,k_1-2},\ldots,Z_{1,0},Z_{2,k_2-1},Z_{2,k_2-2},\ldots$$
is the desired $\sigma$-discrete expansion of $V(G) \cup \Xi$.

First of all, since $V(G) \cup \Xi = \bigcup_{n \in \mathbb{N}} Y_n$ and this union is disjoint, the above sequence has union $V(G) \cup \Xi$. By the definition of rank, it is also clear that all sets in the sequence are discrete. It remains to show that the union over finite initial segments is closed. Clearly, each such union is of the form 
$$Y= X_n \cup Z_{n+1,k_{n+1}-1} \cup \cdots \cup Z_{n+1,i} \subseteq  X_{n+1}$$
for some $i < k_{n+1}$, and this set is closed in $|G|$ as $X_{n+1}$ is closed in $|G|$ and $Y$ is closed in $X_{n+1}$ by the definition of the Cantor-Bendixson rank.

For $(iii)\Rightarrow(iv)$, let $(X'_n)_{n\in\NN}$ be a $\sigma$-discrete expansion for $V(G) \cup \Xi$.
Then we apply Theorem~\ref{thm:topology_to_td} for the closed sets $X_n:=\bigcup_{i \leq n}X'_i$ and the discrete sets $\Delta_n:=X'_n\cap\Omega(G)$ to obtain an upwards connected tree-decomposition of $G$ of finite adhesion into connected parts displaying $\Xi^\complement$ such that all ends from $\Xi = \bigcup_{n\in\NN} \Delta_n$, and hence all ends from the interior of $\mathcal{T}$ live in pairwise distinct parts.
In other words, this tree-decomposition realises $(\Xi^\complement,\Xi)$.

Next, it is clear that $(iv)$ implies $(i)$, which completes the proof.
\end{proof}

We have now all results in place to prove our main result Theorem~\ref{thm_main} from this paper, the following theorem contains even more equivalent properties:

\begin{thm}
\label{thm_mainrestate}
The following are equivalent for any connected graph $G$ with at least one end:
\begin{enumerate}
\item There is an upwards connected tree-decomposition of finite adhesion that represents $\Omega(G)$ such that all parts induce connected subgraphs.
    \item There is a tree-decomposition of finite adhesion that represents all ends in $\Omega(G)$.
    \item There is a tree-decomposition of finite adhesion that distinguishes all ends in $\Omega(G)$.
    \item There is a tree-decomposition of finite adhesion into ${\leq} 1$-ended parts. 
    \item Some subset $\Xi \subseteq \Omega(G)$ of ends can be distributed.
    \item $V(G)$ is a countable union of slender sets.
  \end{enumerate}
\end{thm}

\begin{proof}
The implications $(1) \Rightarrow (2) \Rightarrow (3) \Rightarrow (4) \Rightarrow (5)$  are trivial. The implication $(5) \Rightarrow (6)$ follows from $(i) \Rightarrow (ii)$ in Theorem~\ref{itemDistr}. Finally, for $(6) \Rightarrow (1)$ note that due to $(ii) \Rightarrow (iv)$ in Theorem~\ref{itemDistr}, we immediately get from $(6)$ that there is an upwards connected tree-decomposition of finite adhesion into connected parts that realises $(\Xi,\Xi^\complement)$. But then it follows from the subsequent Lemma~\ref{lem_represent} that there also is such a tree-decomposition $\mathcal{T}'$ that represents some partition $(\Xi',\Psi')$ of $\Omega(G)$ with $\Xi \subseteq \Xi'$, and so $\mathcal{T}'$ represents all ends in $\Omega(G)$ as desired.
\end{proof}

\begin{lemma} \label{lem_represent}
If a connected graph $G$ with at least one end admits a tree-decomposition $\mathcal{T}$ of finite adhesion 
 that realises some partition $(\Xi,\Psi)$ of  $\Omega(G)$, then there also is such a tree-decomposition $\mathcal{T}'$ that represents some partition $(\Xi',\Psi')$ of $\Omega(G)$ with $\Xi \subseteq \Xi'$.
 
 Moreover, whenever $\mathcal{T}$ has connected parts or is upwards connected, we can obtain the same for $\mathcal{T}'$.
\end{lemma}

\begin{proof}
Suppose we are given a tree-decomposition $(T,\cV)$ of finite adhesion realising some partition $(\Xi,\Psi)$ of  $\Omega(G)$. We will perform two rounds of contractions on $T$ to make sure that we represent some partition $(\Xi',\Psi')$ of $\Omega(G)$ with $\Xi \subseteq \Xi'$.

First, pick a maximal family $\cR$ of disjoint rays in $T$ such that no end of $G$ lives in a part corresponding to one of the nodes of a ray in $\cR$. 
Then consider a new tree-decomposition $(\dot{T},\dot{\cV})$ where $\dot{T}$ is obtained from $T$ by contracting each ray in $\cR$. For every $R \in \cR$ we define a corresponding part $\dot{V}_R = \bigcup_{t \in R} V_t$.
Since the set of separators of $(\dot{T},\dot{\cV})$ is a subset of the set of separators of $(T,\cV)$, it follows that also $(\dot{T},\dot{\cV})$ has finite adhesion.
And since by assumption on $(T,\cV)$ there corresponds precisely one end of $G$ to any ray $R \in \cR$, it follows that $(\dot{T},\dot{\cV})$ realises $(\Xi',\Psi')$ where
$\Xi'$ is the union of $\Xi$ together with all ends of $G$ that correspond to a ray in $\cR$, and $\Psi'$ is its complement.  

Next, note that by maximality of $\cR$, every ray of $\dot{T}$ contains infinitely many nodes whose corresponding parts in $\dot{\cV}$ contain an end of $G$. Therefore, if we pick any partition $\cP$ of $V(\dot{T})$ into subtrees such that each subtree $P$ contains a unique node for which there is an end $\omega_P$ of $G$ living in the corresponding part of $\dot{\cV}$, then all $P \in \cP$ are necessarily rayless.

Now consider a new tree-decomposition $(T',\cV')$ where $T'$ is obtained from $\dot{T}$ by contracting each subtree in $\cP$. Naturally, $V(T') = \cP$, and for each $P \in V(T')$ we define $V'_P = \bigcup_{t \in P} \dot{V}_t$.
Since $\cT'$ arises from $\cT$ by contracting subtrees, it is clear that $\cT'$ has finite adhesion, connected parts, or is upwards connected if the same is true for $\cT$.
Lastly, $(T',\cV')$ now represents the partition $(\Xi',\Psi')$, as in each part $V'_P$ there lives precisely the single end $\omega_P$ from $\Xi'$, and since all $P$ were rayless and $(\dot{T},\dot{\cV})$ displays $\Psi'$, also $(T',\cV')$ displays $\Psi'$.
\end{proof}

\begin{cor}
\label{lem_bij}
If a connected graph $G$ with at least one end admits a rayless tree-decomposition $\mathcal{T}$ of finite adhesion that distributes $\Omega(G)$, then there also is such a tree-decomposition that bijectively distributes $\Omega(G)$.
Moreover, whenever $\mathcal{T}$ has connected parts or is upwards connected, we can obtain the same for $\mathcal{T}'$.
\end{cor}

\section{Tree-decompositions distributing all ends}
In the previous section we stated a topological characterisation for the sets of ends that can be distributed. If we are interested in distributing all ends of $G$, we can obtain a combinatorial characterisation in terms of the underlying graph.

The following is a convenient description of the Cantor-Bendixson rank of the space $V \cup \Omega(G) \subseteq |G|$ due to Jung \cite[\S3]{jung1969wurzelbaume}: The \emph{rank} $r(x)$ of a vertex or an end $x$ in a graph $G=(V,E)$ is defined as follows: all vertices have rank $0$.
An end $\omega$ has rank $1$, if there is a finite set $S \subseteq V$, such that $\hat{C}(S,\omega)$ contains no other end.
For an ordinal $\alpha$, we say an end $\omega$ has rank $\alpha$, if it has not already been assigned a smaller rank and if there is a finite set $S \subseteq V$ such that all ends in $\hat{C}(S,\omega)$ have been assigned a rank, and all these ranks are strictly smaller than $\alpha$.

For a graph $G$ in which every end has a rank (i.e.\ for graphs where $V \cup \Omega(G)$ is scattered), we define the \emph{end-rank} $r(G)$ as the supremum of the ranks of all points in $V\cup\Omega(G)$.\footnote{We remark that in this formulation, $r(G)$ and the Cantor-Bendixson rank of $V \cup \Omega(G)$ may differ by $\pm 1$.}

\begin{thm}\label{thm:distributeallends}
The following are equivalent for any connected graph $G$:
\begin{enumerate}[label=(\roman*)]
\item There is an upwards connected rayless tree-decomposition of finite adhesion with connected parts distributing $\Omega(G)$.
    \item There is a tree-decomposition of finite adhesion distributing $\Omega(G)$.
    \item $V \cup \Omega(G)$ has a $\sigma$-discrete expansion.
    \item\label{item:binary_tree} $G$ contains no end-faithful subdivision of the full binary tree $T_2$.
    \item Every end of $G$ has a rank, i.e.\ $\Omega(G)$ is scattered.
 
    \end{enumerate}
Moreover, if $\Omega(G) \neq \emptyset$, we may add    
\begin{enumerate}[label=(\roman*), resume]    
   \item There is an upwards connected rayless tree-decomposition of finite adhesion with connected parts bijectively distributing $\Omega(G)$.
\end{enumerate}
\end{thm} 

\begin{proof}
$(i) \Leftrightarrow (ii) \Leftrightarrow (iii)$ is a special case of Theorem~\ref{itemDistr}.\\
For the implication $(iii) \Rightarrow (iv)$ note that any subspace of $V \cup \Omega(G)$ inherits the property of having a $\sigma$-discrete expansion. However, the end space of a binary tree does not have a $\sigma$-discrete expansion: Indeed, any discrete set in a compact metric space is just countable; but the end space of a binary tree is uncountable, so not a countable union of countable sets. 

The equivalence $(iv) \Leftrightarrow (v)$ is the content of Jung's \cite[Satz~4]{jung1969wurzelbaume}. 

We prove $(v) \Rightarrow (iii)$ by transfinite induction on the end-rank $\alpha$ of $G$.
In the base case $r(G) = 0$, i.e.\ when $\Omega(G) = \emptyset$, we may take the trivial expansion consisting just of the vertex set.

Now let $\alpha > 0$, and suppose that all graphs of rank $<\alpha$ admit a $\sigma$-discrete expansion. First, let $\Phi \subseteq \Omega(G)$ consist of all ends of rank $\alpha$. Clearly, $\Phi$ is a closed discrete subset of $\Omega(G)$.
By Corollary~\ref{cor:envelope}, there is a connected envelope $U$ for $\Phi$, i.e.\ $U$ is a connected set of vertices in $G$ of finite adhesion such that $\partial U = \Phi$. Write $\mathcal{P}$ for the collection of components of $G-U$, and note that for each $P \in \mathcal{P}$, all ends living in $P$ have rank $< \alpha$.

Now for each component $P \in \mathcal{P}$ individually, consider a collection $\mathcal{C}_P = \{ C_P(S_{\omega},\omega):\omega \in \Omega(P) \}$ such that each set $C_P(S_{\omega},\omega)$ witnesses the rank of $\omega$ inside the graph $P$. By Theorem~\ref{thm:RNT}, there is a rayless normal tree $N_P$ in $P$ such that every component $D$ of $P-N_P$ is included in an element of $\mathcal{C}_P$ and hence satisfies $r(D) < \alpha$. Note that $U' = U \cup \bigcup_{P \in \mathcal{P}} N_P$ also is an envelope for $\Phi$, but now, writing $\mathcal{P}'$ for the collection of components of $G-U'$, we have $r(D)<\alpha$ for every $D \in \mathcal{P}'$.
By induction assumption, each $D \in \mathcal{P}'$  admits a $\sigma$-discrete expansion 
$$V(D) \cup \Omega(D) = \bigcup_{n \geq 1} X_{D,n}.$$ 
Then $X_0:=\overline{U'}=U'\cup\Phi$ together with 
$$X_n := \bigcup_{D \in \mathcal{P}'} X_{D,n}$$
for $n\geq 1$ gives the desired $\sigma$-discrete expansion of $V \cup \Omega(G)$. Indeed, to see that $X_0 \cup X_1 \cup \cdots \cup X_n$ is closed for every $n \in \NN$, note that every end $\omega$ of $G$ outside of this set lives in some component $D$ for $D \in \mathcal{P}'$.
Let $S\subseteq V(D)$ be finite such that $\hat{C}_D(S,\omega)$ is a basic open set inside $V(D)\cup\Omega(D)$ separating $\omega$ from the closed set $X_{D,1} \cup \cdots \cup X_{D,n}$.
But then $\hat{C}_G(S\cup N(D),\omega)$ is a basic open neighbourhood of $\omega$ in $V\cup\Omega(G)$ witnessing that $\omega$ does not belong to the closure of $X_0 \cup X_1 \cup \cdots \cup X_n$. This completes the induction step and the proof of $(v) \Rightarrow (iii)$.

Finally, the moreover part $(i) \Leftrightarrow (vi)$ is immediate from Corollary~\ref{lem_bij}.
\end{proof}

Using different methods, Polat showed that $\Omega(G)$ has a $\sigma$-discrete expansion if and only if every end of $G$ has a rank \cite[Theorem~8.11]{polat1996ends2}. 

\section{Applications}
\label{sec_applications}

\subsection{Tree-decompositions displaying special subsets of ends}

Through our main characterisation, we can now give a short proof of the main result from Carmesin's \cite{carmesin2014all}.
\begin{thm}
\label{thm_carmensin_Gdelta}
Every connected graph $G$ has a tree-decomposition of finite adhesion with connected parts that displays precisely the undominated ends of $G$.
\end{thm}

\begin{proof}
Let $\Xi$ be the set of all ends of $G$ which are dominated. By Theorem \ref{thm:display_sets} it suffices to show that $\Omega(G)\setminus\Xi$ is a $G_\delta$ set in $|G|$, and by Lemma \ref{lem_Fsigma} it is equivalent to show that $V(G)\cup\Xi$ is an $F_\sigma$ set in $|G|$. Choose an arbitrary vertex $u\in V(G)$ and for all $n\in\NN$ write $X_n$ for the set of all vertices of $G$ with distance at most $n$ to $u$.
We show that $V(G)\cup\Xi=\bigcup_{n\in\NN}\overline{X_n}$.
We have $V(G)=\bigcup_{n\in\NN}X_n$ because $G$ is connected.
It is left to show that the ends in $\bigcup_{n\in\NN}\overline{X_n}$ are precisely the dominated ends of $G$.

Consider any end $\omega\in\Omega(G)$ and let $R$ be an $\omega$-ray in $G$.
First, suppose that $\omega$ is dominated and let $v$ be the centre of an infinite subdivided star $S$ with leaves in $R$. 
Furthermore, suppose that $v\in X_n$.
Then $S-v$ is a comb attached to $N(v)\subseteq X_{n+1}$ and therefore $\omega$ is contained in $\overline{X_{n+1}}$.

Now assume for a contradiction that some $\overline{X_n}$ contains an undominated end $\omega$, and choose $n$ minimal with that property.
Then there is a comb $C$ attached to $X_n$ with spine $R \in \omega$.
By minimality of $n$, there is an infinite set $\cT$ of teeth of $C$ which lie in $X_n\setminus X_{n-1}$.
The neighbourhood of $\cT$ in $X_{n-1}$ is finite, again by minimality of $n$.
Since every vertex in $X_n\setminus X_{n-1}$ has a neighbour in $X_{n-1}$, there is vertex $v\in X_{n-1}$ with infinitely many neighbours in $\cT$.
Hence $\omega$ is dominated by $v$, a contradiction.
\end{proof}

The following generalises a corresponding result from \cite[Theorem~2]{burger2022duality}.

\begin{thm}
\label{thm_carljan_Gdelta}
For every infinite set of vertices $U$ in a connected graph $G$, there is a tree-decomposition of $G$ of finite adhesion that displays precisely the undominated ends of $\partial U$.
\end{thm}

\begin{proof}
Without loss of generality, we may assume that $U$ has finite adhesion (Theorem~\ref{thm:envelope}).

Consider the contraction minor $H \preceq G$ obtained from $G$ by contracting each component $C$ of $G - U$ to a single vertex $v_C$ (of finite degree). 

\begin{clm}
\label{claim_jan1}
The inclusion $U \hookrightarrow H$ induces a bijection $\partial U \to \Omega(H)$ that preserves the property of being dominated.
\end{clm}

This claim is proven just like Lemma~\ref{lem_torso}.

\begin{clm}
\label{claim_jan2} The contractions resulting in $H$ induce a natural continuous surjection $f \colon |G| \to |H|$.
\end{clm}

To see that $f$ is continuous, consider some end $\omega \in |G|$. If $\omega \notin \partial U$, then $f(\omega) = v_C$ for some component $C$, and $f$ is continuous at $\omega$. If $\omega \in \partial U$, then $f(\omega) = \omega' \in \Omega(H)$ by Claim~\ref{claim_jan1}. Let $C_H(X',\omega')$ be an arbitrary basic open neighbourhood around $\omega'$ in $H$. Let $X \subseteq U$ be the finite set of vertices where we replace every vertex of the form $v_C$ in $X'$ by $N(C)$. It remains to verify that 
$$f[C_G(X,\omega)] \subseteq C_H(X',\omega').$$
But this is clear: for every $v \in C_G(X,\omega)$, any $v-\omega$-ray $R$ avoiding $X$ is mapped to a locally finite connected subgraph in $H$ avoiding $X'$ which includes an $f(v)-\omega'$-ray $R'$.

Now we apply Theorem~\ref{thm_carmensin_Gdelta} inside $H$ to see that there is a tree-decomposition of finite adhesion displaying the undominated ends $\Psi$ of $H$. Hence $\Psi$ is $G_\delta$ in $|H|$ by Theorem \ref{thm:display_sets}, say $\Psi = \bigcap_{n \in \NN} O_n$ with $O_n$ open in $|H|$. But then by Claim~\ref{claim_jan2},
$$f^{-1}(\Psi) = f^{-1}(\bigcap_{n \in \NN} O_n) = \bigcap_{n \in \NN} f^{-1}(O_n) $$ is $G_\delta$ in $|G|$. Thus $f^{-1}(\Psi)$ can be displayed by a tree-decomposition of finite adhesion of $G$, again by Theorem \ref{thm:display_sets}.
This completes the proof as $f^{-1}(\Psi)$ is the set of all undominated ends in $\partial U$ by Claim~\ref{claim_jan1}.
\end{proof}

\subsection{Counterexamples}
Consider the full infinite binary tree $T_2$, and let $X \subseteq \Omega(T_2)$ be any set of ends. A \emph{binary tree with tops $X$} is the graph with vertex set $T_2 \sqcup X$, all edges of $T_2$, and such that the neighbourhood of $x \in X$ consists of infinitely many nodes on its corresponding normal ray in $T_2$. 

We reobtain Carmesin's observation that a $T_2$ with uncountably many tops does not admit a tree-decomposition of finite adhesion displaying all its ends, but now with significantly shorter proof.

\begin{exmp}
No binary tree with uncountably many tops admits a tree-decomposition of finite adhesion displaying all its ends.
\end{exmp}
\begin{proof}
These graphs do not have normal spanning trees by \cite[Proposition~3.3]{diestel2001normal}, and so the result follows from Theorem~\ref{thm_Psi=Omega}.
\end{proof}

With only a little more work, we can prove the following stronger result by Carmesin \cite[p.7]{carmesin2014all}.

\begin{exmp}
No binary tree with uncountably many tops admits a tree-decomposition of finite adhesion  distinguishing all its ends.
\end{exmp}

\begin{proof}
Let $G$ be a binary tree with uncountably many tops. Suppose for a contradiction that $V(G)$ is a countable union of slender sets.
Then one of the slender sets $U$ contains uncountably many of the tops.
Write $\cR$ for the set of all normal rays of $T_2$ which have a corresponding top in $U$.
We call a vertex $v$ of $T_2$ \emph{good}, if it lies in uncountably many rays from $\cR$.
It is clear that the root of $T_2$ is good. We now show that for each good vertex $v$, there are two incomparable good vertices above $v$ in the tree-order:

Suppose not for a contradiction.
It is clear that at least one upper neighbour in $T_2$ of each good vertex is good. This implies that there is a ray $R$ of good vertices above $v$.
Since per assumption all good vertices above $v$ are comparable, no other vertex above $v$ outside the ray $R$ is good.
But this ray has only countable many neighbours in $T_2$. As no such neighbour above $v$ is  good, every neighbour of $R$ above $v$ lies on only countably many rays from $\cR$. But then also $v$ lies on only countably many rays from $\cR$, which is a contradiction since $v$ is good.

From this claim follows that there is a subdivided binary tree inside $G$ such that each branch vertex is good.

It follows that $\partial U$ itself contains the end space of a subdivided binary tree. But the end space of a binary tree is not scattered, a contradiction.
It follows from Theorem \ref{thm_mainrestate} that $\Omega(G)$ cannot be distinguished.
\end{proof}

We conclude this section with a new example of a graph $G$ witnessing that the thin ends of $G$ cannot always be displayed, that is based on topological considerations only (a different example is given by Carmesin in \cite[Example~3.3]{carmesin2014all}).  More precisely, since displayable subsets of ends are always completely metrizable by Corollary~\ref{end-space-of-tree}, it suffices to construct a graph where the thin ends are not completely metrizable.

As a warm-up, consider the binary tree $T$, and call a normal ray of $T$ \emph{rational} if its corresponding $0-1$-sequence becomes eventually constant, and \emph{irrational} otherwise. Let $\Sigma \subseteq \Omega(T_2)$ be the subspace of rational ends. By Sierpinski's characterisation \cite{sierpinski1920propriete}, every countable metric space without isolated points -- so in particular $\Sigma$ -- is homeomorphic to the rational numbers $\mathbb{Q}$. Thus, $\Sigma$ is not completely metrizable, and hence not displayable.

We now modify $T$ such that all irrational ends  become thick, and all rational ends remain thin. A binary tree with \emph{fat} tops $Z$ is a graph with vertex set $T \sqcup Z$, all edges of $T_2$, and such that the neighbourhood of $z \in Z$ consists of infinitely many nodes on some normal ray $R_z$ of $T_2$. Thus, the difference between a tree with tops and tree with fat tops is that a normal ray may now have more than one top vertex.

\begin{exmp}\label{thin_ends}
There is a binary tree with uncountably many fat tops such that its thin ends cannot be displayed.
\end{exmp}

\begin{proof}[Construction]
Starting from the binary tree $T$, let  $\{R_i \colon i \in \NN\}$ be an enumeration of the rational rays in $T$. We now add infinitely many top-vertices above each irrational ray, and connect them to their rays such that
\begin{enumerate}
  \item each top-vertex $z$ dominates its corresponding irrational ray $R_z$, and
  \item for each rational ray $R_i$, at most $i$ vertices on $R_i$ have top-vertices as neighbours. 
\end{enumerate}
Once the construction has finished, it is clear that the resulting graph $G$ is as desired. The end space $\Omega(G) = \Omega(T)$ remains unchanged. 
From (2) it is easy to see that every rational end $\omega_i \ni R_i$ has end-degree $1$ in $G$ (and hence is thin), since the corresponding ray $R_i$ has a tail of vertices of degree $3$ whose edges are cut edges.
All irrational ends are dominated by their infinitely many top-vertices and thus become thick.

By Sierpinski's characterisation \cite{sierpinski1920propriete}, the set of rational / thin ends is homeomorphic to the rational numbers $\mathbb{Q}$, so not completely metrizable, and hence not displayable by Corollary~\ref{end-space-of-tree}.

It remains to describe how to connect a top-vertex $z$ to its irrational ray $R_z$.
For each top-vertex $z$ and every $j \in \NN$, let $r_z^j$ be the $\leq_T$-minimal vertex in $R_z \setminus (R_0 \cup \ldots \cup R_j)$.
Now let the neighbours of $z$ be exactly the vertices in $\{r_z^0,r_z^1,r_z^2,\ldots\}$.
Since $r_z^0\leq r_z^1 \leq r_z^2 \leq \cdots$ is cofinal in $R_z$, the top-vertex $z$ dominates $R_z$, establishing property (1).

Next, consider the $i$th rational ray $R_i$. Again, for $j \leq i$ let $r_i^j$ be the $\leq_T$-minimal vertex in $R_i \setminus (R_0 \cup \ldots \cup R_j)$ (if it exists). Then it is clear that $r_i^0,r_i^1,\ldots,r_i^{i-1}$ are the only vertices on $R_i$ adjacent to top-vertices, giving (2).
\end{proof}

\bibliographystyle{plain}
\bibliography{ref}

\begin{bibdiv}
\begin{biblist}

\bib{bowler2013infinite}{article}{
      author={Bowler, Nathan},
      author={Carmesin, Johannes},
       title={Infinite matroids and determinacy of games},
        date={2013},
     journal={arXiv preprint arXiv:1301.5980},
}

\bib{bowler2013ubiquity}{article}{
      author={Bowler, Nathan},
      author={Carmesin, Johannes},
       title={The ubiquity of psi-matroids},
        date={2013},
     journal={arXiv preprint arXiv:1304.6973},
}

\bib{StarComb1StarsAndCombs}{article}{
      author={B{\"u}rger, Carl},
      author={Kurkofka, Jan},
       title={{Duality theorems for stars and combs {I}: Arbitrary stars and
  combs}},
        date={2022},
     journal={Journal of Graph Theory},
      volume={99},
      number={4},
       pages={525\ndash 554},
}

\bib{burger2022duality}{article}{
      author={B{\"u}rger, Carl},
      author={Kurkofka, Jan},
       title={Duality theorems for stars and combs {II}: Dominating stars and
  dominated combs},
        date={2022},
     journal={Journal of Graph Theory},
      volume={99},
      number={4},
       pages={555\ndash 572},
}

\bib{carmesin2014all}{article}{
      author={Carmesin, J.},
       title={All graphs have tree-decompositions displaying their topological
  ends},
        date={2019},
     journal={Combinatorica},
      volume={39},
      number={3},
       pages={545\ndash 596},
}

\bib{carmesin2017topological}{article}{
      author={Carmesin, Johannes},
       title={Topological cycle matroids of infinite graphs},
        date={2017},
     journal={European Journal of Combinatorics},
      volume={60},
       pages={135\ndash 150},
}

\bib{diestel1992end}{article}{
      author={Diestel, R.},
       title={The end structure of a graph: recent results and open problems},
        date={1992},
     journal={Discrete Mathematics},
      volume={100},
      number={1--3},
       pages={313\ndash 327},
}

\bib{diestel1994depth}{article}{
      author={Diestel, R.},
       title={{The depth-first search tree structure of $TK_{\aleph_0}$-free
  graphs}},
        date={1994},
     journal={Journal of Combinatorial Theory, Series B},
      volume={61},
      number={2},
       pages={260\ndash 262},
}

\bib{diestel2006end}{article}{
      author={Diestel, R.},
       title={End spaces and spanning trees},
        date={2006},
     journal={J. Combin.\ Theory (Series B)},
      volume={96},
      number={6},
       pages={846\ndash 854},
}

\bib{diestel2011locally}{article}{
      author={Diestel, R.},
       title={Locally finite graphs with ends: A topological approach,
  {I}--{III}.},
        date={2010/11},
     journal={Discrete Mathematics},
      volume={311--312},
}

\bib{diestel2015book}{book}{
      author={Diestel, R.},
       title={{Graph Theory}},
     edition={5},
   publisher={Springer},
        date={2015},
}

\bib{diestel2003graph}{article}{
      author={Diestel, R.},
      author={Kühn, D.},
       title={{Graph-theoretical versus topological ends of graphs}},
        date={2003},
     journal={J. Combin.\ Theory (Series B)},
      volume={87},
      number={1},
       pages={197\ndash 206},
}

\bib{diestel2001normal}{article}{
      author={Diestel, Reinhard},
      author={Leader, Imre},
       title={Normal spanning trees, {A}ronszajn trees and excluded minors},
        date={2001},
     journal={Journal of the London Mathematical Society},
      volume={63},
      number={1},
       pages={16\ndash 32},
}

\bib{diestel2017dual}{article}{
      author={Diestel, Reinhard},
      author={Pott, Julian},
       title={Dual trees must share their ends},
        date={2017},
     journal={Journal of Combinatorial Theory, Series B},
      volume={123},
       pages={32\ndash 53},
}

\bib{Torso}{article}{
      author={Elm, A.},
      author={Kurkofka, J.},
       title={A tree-of-tangles theorem for infinite tangles},
        date={2020},
     journal={arXiv preprint arXiv:2003.02535},
}

\bib{engelking1989book}{book}{
      author={Engelking, Ryszard},
       title={General toplogy -- revised and completed ed.},
   publisher={Heldermann Verlag Berlin},
        date={1989},
      volume={6},
}

\bib{jung1969wurzelbaume}{article}{
      author={Jung, H.A.},
       title={{Wurzelb{\"a}ume und unendliche Wege in Graphen}},
        date={1969},
     journal={Math.\ Nachr.},
      volume={41},
       pages={1\ndash 22},
}

\bib{FirstSecondCountable}{misc}{
      author={Kurkofka, J.},
      author={Melcher, R.},
       title={Countably determined ends and graphs},
        date={(2021)},
        note={To appear in Journal of Combinatorial Theory, Series B},
}

\bib{kurkofkapitz_rep}{misc}{
      author={Kurkofka, J.},
      author={Pitz, M.},
       title={A representation theorem for end spaces},
        date={(2021)},
        note={Submitted.},
}

\bib{kurkofka2021approximating}{article}{
      author={Kurkofka, Jan},
      author={Melcher, Ruben},
      author={Pitz, Max},
       title={Approximating infinite graphs by normal trees},
        date={2021},
     journal={Journal of Combinatorial Theory, Series B},
      volume={148},
       pages={173\ndash 183},
}

\bib{pitz_shortCarmesin}{misc}{
      author={Pitz, M.},
       title={Constructing tree-decompositions that display all topological
  ends},
        date={(2021)},
        note={To appear in Combinatorica.},
}

\bib{pitz2021proof}{article}{
      author={Pitz, Max},
       title={Proof of {H}alin’s normal spanning tree conjecture},
        date={2021},
     journal={Israel Journal of Mathematics},
      volume={246},
      number={1},
       pages={353\ndash 370},
}

\bib{polat1996ends}{article}{
      author={Polat, Norbert},
       title={Ends and multi-endings, {I}},
        date={1996},
     journal={Journal of Combinatorial Theory, Series B},
      volume={67},
       pages={86\ndash 110},
}

\bib{polat1996ends2}{article}{
      author={Polat, Norbert},
       title={Ends and multi-endings, {II}},
        date={1996},
     journal={Journal of Combinatorial Theory, Series B},
      volume={68},
       pages={56\ndash 86},
}

\bib{sierpinski1920propriete}{article}{
      author={Sierpi{\'n}ski, Wac{\l}aw},
       title={Sur une propri{\'e}t{\'e} topologique des ensembles
  d{\'e}nombrables denses en soi},
        date={1920},
     journal={Fundamenta Mathematicae},
      volume={1},
      number={1},
       pages={11\ndash 16},
}

\bib{thomas1989well}{article}{
      author={Thomas, Robin},
       title={Well-quasi-ordering infinite graphs with forbidden finite planar
  minor},
        date={1989},
     journal={Transactions of the American Mathematical Society},
      volume={312},
      number={1},
       pages={279\ndash 313},
}

\end{biblist}
\end{bibdiv}
\end{document}